\documentclass[a4paper,11pt,DIV=11]{scrartcl}
\usepackage[utf8]{inputenc}
\usepackage[english]{babel}
\usepackage{csquotes}
\usepackage{enumerate,xparse}
\usepackage{soul}
\usepackage{amsmath, amssymb, amsfonts, amsthm,mathtools}
\usepackage{algorithm,algorithmic}
\usepackage{graphicx}
\usepackage[small,labelfont=bf]{caption}
\usepackage{subcaption}
\usepackage{color} % Colors
\usepackage{siunitx,booktabs,pgfplotstable}
\usepackage[breaklinks=true]{hyperref}
\usepackage[
  style=trad-alpha,
  sorting=nyt,
  url=false,
  doi=true,                % Print the DOI.
  eprint=true,
  giveninits=true,         % Abbreviate first names of the authors.
  uniquename=false,        % Disable this feature. Print
  maxbibnames=99,          % Do not use ``et al'' in the bibliography
  sortcites=true,         % Do not sort if multiple keys in one \cite{}
  backend=biber,           % Use biber as backend
  safeinputenc,            % For some accents, see http://tex.stackexchange.com/questions/170562
  isbn=false,              % Do not shown ISBN, ISSN, ISRN
]{biblatex}
\def\oaitourl#1{http://hal.archives-ouvertes.fr/#1}
\DeclareFieldFormat{eprint:oai}{HAL:~%
  \ifhyperref{\href{\oaitourl#1}{\nolinkurl{#1}}}%
    {\nolinkurl{#1}}%
}
\ifpdf
\hypersetup{
  pdfauthor={Ronny Bergmann, Pierre-Yves Gousenbourger},
  pdftitle={  A variational model for data fitting on manifolds by minimizing the acceleration of a Bézier curve},
  pdfsubject={preprint},%
  pdfcreator = {pdflatex and TextMate},%
  pdfkeywords={},
  plainpages=false, pdfstartview=FitH, pdfview=FitH, pdfpagemode=UseOutlines,%
  bookmarksnumbered=true,bookmarksopen=false,bookmarksopenlevel=0,%
  colorlinks=true,
  linkcolor=blue!75!black,
  urlcolor=green!50!black,
  citecolor=blue!50!green!62!black
}
\fi
\definecolor{darkgreen}{rgb}{0,0.6,0.1}
\definecolor{darkblue}{rgb}{0,0,0.7}
\definecolor{mygreen}{rgb}{0.3,0.8,0}
\definecolor{myblue}{rgb}{0,0.5,1}
\definecolor{myred}{rgb}{1,0,0}

\DeclareMathOperator*{\argmin}{\arg\!\min}

\newcommand{\R}        {\mathbb{R}}              % euclidean space
\newcommand{\N}        {\mathbb{N}}              % Naturals
\newcommand{\M}        {\mathcal{M}}              % Manifold
\newcommand{\C}       {\mathcal{C}}              % Continuity
\renewcommand{\d}         {\mathrm{d}}              % d straight
\newcommand{\bc}[1]     {{[#1]}}
\newcommand{\defTerm}[1]    {\emph{#1}}                             % definition of terms

\newcommand{\tT}{\mathrm{T}}

\makeatletter
\DeclareDocumentCommand\Exp{ m o o}{%
\def\@tempa{#3}
\operatorname{exp}_{#1}%
\IfNoValueF{#2}{\!%
\ifx\@tempa\@empty\relax \else
\IfNoValueTF{#3}{\left(}{\mathopen#3(} \fi%
#2%
\ifx\@tempa\@empty\relax \else
\IfNoValueTF{#3}{\right)}{\mathclose#3)} \fi%
}%
}%
\DeclareDocumentCommand\Log{ m o o}{%
  \mathrm{log}_{#1}%
\IfNoValueF{#2}{\!%
\ifx\@tempa\@empty\relax \else
\IfNoValueTF{#3}{\left(}{\mathopen#3(} \fi%
#2%
\ifx\@tempa\@empty\relax \else
\IfNoValueTF{#3}{\right)}{\mathclose#3)} \fi%
}%
}%
\makeatother
\newcommand{\bspline}    {\mathbf{B}}              % B\'ezier spline

\newcommand{\vect}[1]{\mathbf{#1}}

\newcommand{\ie}{{\emph{i.e.}}}
\newcommand{\eg}{{\emph{e.g.}}}

\newcommand{\etal}{{\emph{et al.}}}

  \tikzstyle{link}=[->,>=stealth,thick]
  \tikzset{
    item/.style={rectangle,draw,thick},
    link0/.style={ dashed, line width=1pt },
    link1/.style={ line width=2pt },
    boxed/.style={rectangle,draw},
    jacobi/.style={<-,line width = 1pt},
    revjacobi/.style={<-,line width = 1pt,dashed},
    }

\theoremstyle{plain}
\newtheorem{thm}{Theorem}       % theorems
\newtheorem{defn}[thm]{Definition}   % Definitions
\newtheorem{lem}[thm]{Lemma}     % Lemma
\newtheorem{exmp}[thm]{Example}   % Examples (italic)

\theoremstyle{remark}
\newtheorem*{remark}{Remark}      % remarks

\addtokomafont{disposition}{\rmfamily}
\addtokomafont{title}{\normalfont}

\title{
  A variational model for data fitting on manifolds\\ by minimizing the acceleration of a B\'ezier curve
}

\author{
  Ronny~Bergmann\footnotemark[1]
  \and
  Pierre-Yves~Gousenbourger\footnotemark[2]
}

\date{July 26, 2018}

\begin{document}

% Title
\maketitle
\renewcommand{\thefootnote}{\fnsymbol{footnote}}
\footnotetext[1]{Research Group Numerical Mathematics
  (Partial Differential Equations),
    Faculty of Mathematics,\\
    Technische Universität Chemnitz,
    D-09107 Chemnitz, Germany \\
    E-mail: \href{mailto:ronny.bergmann@mathematik.tu-chemnitz.de}{\texttt{ronny.bergmann@mathematik.tu-chemnitz.de}}}
\footnotetext[2]{ICTEAM Institute,\\
  Université catholique de Louvain,
  B-1348 Louvain-la-Neuve, Belgium \\
    E-mail: \href{mailto:pierre-yves.gousenbourger@uclouvain.be}{\texttt{pierre-yves.gousenbourger@uclouvain.be}}}
\renewcommand{\thefootnote}{\arabic{footnote}}

% Abstract
\begin{abstract}
  \noindent
  We derive a variational model to fit a composite Bézier curve to a set of
  data points on a Riemannian manifold. The resulting curve 
  is obtained in such a way that its mean squared acceleration is minimal 
  in addition to remaining close the data points.
  We approximate the acceleration 
  by discretizing the squared second order derivative along the curve. 
  We derive a closed-form, numerically stable and efficient
  algorithm to compute the gradient of a Bézier curve on manifolds
  with respect to its control points, expressed as a concatenation of so-called
  adjoint Jacobi fields.
  Several examples illustrate the capabilites and validity of this
  approach both for interpolation and approximation. The examples also
  illustrate that the approach outperforms previous works tackling this problem.
\end{abstract}
\section{Introduction}
  % The problem
  This papers addresses the problem of fitting a smooth curve to
  data points $d_0,\dots,d_n$ lying on a Riemannian manifold $\M$ and
  associated with real-valued parameters $t_0,\dots,t_n$. The curve
  strikes a balance between a data proximity constraint and a smoothing
  regularization constraint.

  % The motivation
  Several applications motivate this problem in engineering and the sciences.
  For instance, curve fitting is of high interest in
  projection-based model order reduction of one-dimensional dynamical
  systems~\cite{Pyta2016}. In that application, the dynamical
  system depends on the Reynolds number and the model reduction is
  obtained by computing suitable projectors as points on a Grassmann
  manifold. Finding a projector is however a time- and memory-consuming
  task. Based on projectors precomputed for given parameter values,
  fitting is used to approximate the projector associated with a new
  parameter value.  In~\cite{Gousenbourger2017}, the same strategy is
  used to approximate wind field orientations represented as points on
  the manifold of positive semidefinite covariance matrices of size $p$ and
  rank $r$. Further applications are involving rigid-body motions on
  $\mathrm{SE}(3)$, like Cosserat rods applications~\cite{Sander2010}, 
  or orientation tracking~\cite{Park2010}. Sphere-valued data are also of
  interest in many applications of data analysis, for storm tracking and
  prediction, or the study of bird migration~\cite{Su2014}.

  % Literature
  There exists different approaches to tackle the curve fitting problem.
  Among others, we name here the subdivision schemes
  approach~\cite{Dyn2009,Wallner2007} or the Lie-algebraic
  methods~\cite{Shingel2008}.
  However, the most popular approach nowadays is probably to encapsulate
  the two above-mentioned constraints into an optimization problem
  \begin{equation}
  \label{eq:E}
  \min_{\gamma \in \Gamma} E_\lambda(\gamma) \coloneqq
    \int_{t_0}^{t_n} \Bigl\lVert
      \frac{\mathrm D^2 \gamma(t)}{\d t^2}
    \Bigr\rVert^2_{\gamma(t)}
    \d t
    +
    \frac{\lambda}{2} \sum_{i=0}^n \d^2\big(\gamma(t_i),d_i)\big),
  \end{equation}
  where $\Gamma$ is an admissible space of curves $\gamma\colon
  [t_0,t_n] \to \M$, $t \mapsto \gamma(t)$,
  $\frac{\mathrm D^2}{\d t^2}$ denotes the (Levi-Civita) second
  covariant derivative,
  $\lVert\cdot \rVert_{\gamma(t)}$ denotes the Riemannian metric at $\gamma(t)$
  from which we can define a Riemannian distance $\d(\cdot,\cdot)$.
  Finally, $\lambda \in \R$ is a parameter that strikes the balance between the
  \emph{regularizer}
  $\displaystyle\int_{t_0}^{t_n} \Bigl\lVert\frac{\mathrm D^2
  \gamma(t)}{\d t^2} \Bigr\rVert^2_{\gamma(t)} \d t$ and the
  \emph{fitting} term $\displaystyle\sum_{i=0}^n \d^2(\gamma(t_i),d_i))$.
  This approach leads to the remarkable property that, when the
  manifold $\M$ reduces to the Euclidean space and $\Gamma$ is the
  Sobolev space $H^2(t_0,t_n)$, the solution to~\eqref{eq:E} is the natural
  cubic spline~\cite{Green1993}.
  
  Optimization on manifolds has gained a lot of interest this last decade, 
  starting with the textbook~\cite{Absil2008} that summarizes several optimization
  methods on matrix manifolds. Recently, toolboxes have emerged, providing 
  easy access to such optimization methods, \eg, Manopt~\cite{Manopt} and
  MVIRT~\cite{Bergmann2017}.
  The former received a very positive return in many different topics of
  research, with applications for example in low-rank
  modelling in image analysis~\cite{Zhou2015},
  dimensionality reduction~\cite{Cunningham2015}, phase retrieval~\cite{Sun2017}
  or even 5G-like MIMO systems~\cite{Yu2016}.
  The latter stems from recent interest in manifold-valued image and data
  processing, phrased as variational models on the product manifold \(\M^N\),
  where \(N\) is the number of pixels. Starting with total variation (TV)
  regularization of phase-valued data~\cite{Strekalovskiy2011,Strekalovskiy2013},
  different methods for TV
  on manifolds~\cite{Lellmann2013,Weinmann2014} have been developed as well as
  second order methods~\cite{Bergmann2014,Bacak2016} up to infimal
  convolution~\cite{Bermann2017a} and total generalized
  variation~\cite{Bermann2017a,Bredies2018}.
  Furthermore, different algorithms have been generalized to manifold-valued
  data, besides the previous works using gradient descent or cyclic proximal point
  methods, a Douglas–Rachford splitting~\cite{Bergmann2016a}, iteratively reweighted
  least squares~\cite{Grohs2014} and more general half-quadratic
  minimization~\cite{Bergmann2016} have been introduced.

  The curve fitting problem~\eqref{eq:E}, has been tackled differently the past few
  years. Samir \etal~\cite{Samir2012} considered the case where $\Gamma$ is an
  infinite dimensional Sobolev space of curves, and used the Palais-metric to
  design a gradient descent algorithm for~\eqref{eq:E} (see, \eg, \cite{Su2012} 
  for an application of this approach).
  Another method consists in discretizing the curve $\gamma$ in $N$ points,
  and therefore consider $\Gamma = \M^N$ (see, \eg, \cite{Boumal2011} for
  a result on $\mathrm{SO}(3)$).
  Finally, the limit case where $\lambda \to 0$ is already well studied and
  known as the geodesic regression~\cite{Kim2018,Fletcher2013,Rentmeesters2011}.

  % The Bezier approach
  A recent topic concerns curve fitting by means of Bézier curves.
  In that approach, the search space $\Gamma$ is reduced to as set of
  composite Bézier curves. Those are a very versatile tool to model
  smooth curves and surfaces for real- and vector-valued discrete data
  points (see~\cite{Farin2002} for a comprehensive textbook), but they
  can also be used to model smooth curves and surfaces for manifold-valued
  data~\cite{Absil2016,Popiel2007}. The advantage to work with such objects,
  compared to classical approaches, are that (i) the search space is
  drastically reduced to the so-called control points of the Bézier curves
  (and this leads to better time and memory performances) and (ii) it is
  very simple to impose differentiability for the optimal curve, which is
  appreciated in several of the above-mentioned applications.
  However, while obtaining such an optimal curve reduces directly to
  solving a linear system of equations for data given on a Euclidean space,
  there is up to now no known closed form of the optimal Bézier curve for
  manifold valued data.

  % Our approach
  In this work, we derive a gradient descent algorithm to compute a
  differentiable composite Bézier curve $\bspline\colon[t_0,t_n]\to\M$ that 
  satisfies~\eqref{eq:E}, \ie, such that $\bspline(t)$ has a minimal 
  mean squared acceleration, and fits the set of $n+1$ manifold-valued 
  data points at their associated time-parameters. 
  We consider the manifold-valued generalization of Bézier curves~\cite{Popiel2007}
  in the same setting as in Arnould~\etal~\cite{Arnould2015}, 
  or more recently in~\cite{Gousenbourger2018}.
  We employ the (squared) second order absolute differences introduced
  in~\cite{Bacak2016} to obtain a discrete approximation of the regularizer
  from~\eqref{eq:E}. The quality of the approximation depends only on the number
  of sampling points. We exploit the recursive structure of the De Casteljau 
  algorithm~\cite{Popiel2007} to derive the gradient of the objective function 
  with respect to the control points of~$\bspline(t)$. The gradient is built
  as a recursion of Jacobi fields that, for numerical reasons, are implemented
  as a concatenation of so-called \defTerm{adjoint Jacobi fields}. Furthermore,
  the corresponding variational model only depends on the number of control 
  points of the composite Bézier curve, and not on the number of sampling points. 
  We finally obtain an approximating model to~\eqref{eq:E} that we
  solve with an algorithm only based on three tools on the manifold:
  the exponential map, the logarithmic map, and a certain Jacobi field along geodesics.
  
  % Outline
  The paper is organized as follows. We introduce the necessary
  preliminaries ---Bézier curves, Riemannian manifolds and Riemannian second
  order finite differences--- in Section~\ref{sec:Preliminaries}.
  In Section~\ref{sec:Gradient} we derive the 
  gradient of the discretized mean squared acceleration of the composite
  Bézier curve with respect to its control points,
  and thus of the regularizer of~\eqref{eq:E}.
  In Section~\ref{sec:Numerics}, we present the corresponding gradient descent 
  algorithm, as well as an efficient gradient evaluation method, to
  solve~\eqref{eq:E} for different values of $\lambda$. The limit case
  where~$\lambda \to \infty$ is studied as well.
  Finally, in Section~\ref{sec:Examples}, we validate, analyze and illustrate
  the performance of the algorithm for several numerical examples on the
  sphere~$\mathbb{S}^2$ and on the special orthogonal group $\mathrm{SO}(3)$.
  We also compare our solution to existing B\'ezier fitting methods.
  A conclusion is given in Section~\ref{sec:Concl}.

%
% --- Preliminaries
\section{Preliminaries}\label{sec:Preliminaries}
  %
  % Bezier Curves
  \subsection{Bézier functions and composite Bézier spline}
  \label{subsec:BezierRn}
    Consider the Euclidean space $\R^m$. A \defTerm{Bézier curve of degree
    $K\in\N$} is a function $\beta_K\colon[0,1]\to\R^m$ parametrized by
    \defTerm{control points} $b_0,\ldots,b_K \in \R^m$ and taking the form
    \begin{equation*}
      \label{eq:bezier_eucl}
      \beta_K(t ; b_0, \ldots, b_K)
      \coloneqq
      \sum_{j=0}^K b_j B_{j,K}(t),
    \end{equation*}
    where \(B_{j,K}(t) = \binom{K}{j}t^j (1-t)^{K-j}\) are the \defTerm{Bernstein basis
    polynomials}~\cite{Farin2002}.
    For example the linear Bézier curve~\(\beta_1\) is just the line segment
    \((1-t)b_0 + tb_1\) connecting the two control points \(b_0\) and \(b_1\). 
    The explicit formulae of the quadratic and
    cubic Bézier curves read
    \begin{align*}
      \beta_2(t ; b_0, b_1, b_2) &= b_0(1-t)^2 + 2b_1(1-t)t + b_2 t^2,\\
      \beta_3(t ; b_0, b_1, b_2, b_3)
        &= b_0(1-t)^3 + 3 b_1(1-t)^2t + 3 b_2(1-t)t^2 + b_3 t^3,
    \end{align*}
    for given control points \(b_0,b_1,b_2\in\mathbb R^m\) and
    an additional point~\(b_3\in\mathbb R^m\) for the cubic case.

    A \defTerm{composite Bézier curve} is a function
    \(\bspline\colon[0,n]\to\R^m\) composed of \(n\) Bézier curves and defined
    as
    \begin{equation}\label{eq:RnCompBezier}
      \bspline(t) \coloneqq
        \begin{cases}
          \beta_{K_0}(t; b_0^0,\ldots,b_{K_0}^0)&\mbox{ if } t\in [0,1]\\
          \beta_{K_i}(t-i; b_0^i,\ldots,b_{K_i}^i)&\mbox{ if }
            t\in (i,i+1], \quad i\in\{1,\ldots,n-1\},
        \end{cases}
    \end{equation}
    where \(K_i\in\N\) denotes the degree of the \(i^\text{th}\) 
    Bézier curve~\(\beta_{K_i}\)~of $\bspline$ and~\(b_j^i\), \(j=0,\dots,K_i\),
    are its control points. Furthermore, $\bspline(t)$ is continuous if the last
    and first control points of two consecutive segments
    coincide~\cite{Farin2002}. We introduce~\( p_i \) as the point at
    the junction of two consecutive Bézier segments, \ie, \(p_i\coloneqq
    b_{K_{i-1}}^{i-1} = b_0^{i}\). Differentiability is obtained if the control
    points of two consecutive segments are aligned.
    We introduce further~\( b_{i+1}^- \coloneqq b_{K_{i}-1}^i\) and~\(
    b_{i+1}^+ \coloneqq b_1^{i+1} \) such that the differentiability condition
    reads~\( p_{i+1} = \frac{K_{i} b_{i+1}^- + K_{i+1} b^{i+1}_+}{K_{i} +
    K_{i+1}} \).

    \begin{exmp}[\(C^1\) conditions for composite cubic Bézier curves]
      We consider the case where the Bézier segments are all cubic, as
      represented in Fig.~\ref{fig:bezier}. The composite Bézier
      curve \( \bspline(t) \) is \( C^1 \) if \( p_{i} = \frac{1}{2}(b_i^- +
      b_i^+)\), \( i = 1,\dots,4 \), with \(p_i = b_3^{i-1} = b_0^i\), \(
      b_i^- = b_2^{i-1} \), and \( b_i^+ = b_1^i \).
    \end{exmp}

    \begin{figure}[tbp]
      \centering
      \includegraphics{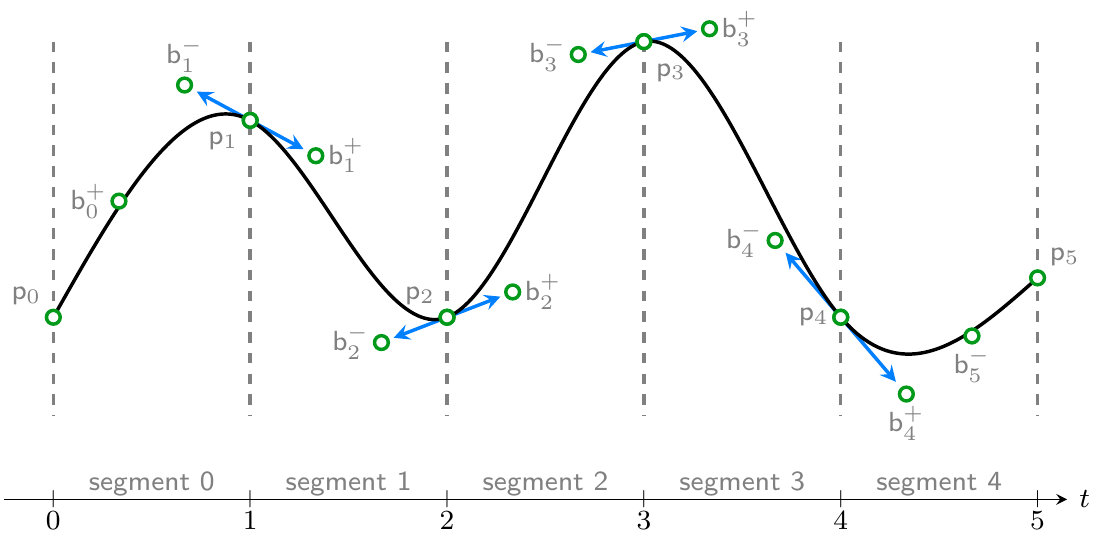}
      \caption{Schematic representation of the composite cubic Bézier
          curve \(\bspline: [0,5] \to \M\), for \(\M = \R\) (black), and
          its control points (green circles).
          Continuous differentiability is reached
          at the junction of the segments (the blue arrows
          draw the first derivative of \(\bspline\)).}
         \label{fig:bezier}
    \end{figure}
  %
  % Riemannian Manifolds
  \subsection{Riemannian Manifolds}
    We consider a complete \(m\)-dimensional Riemannian manifold $\M$. We refer
    to~\cite{ONeill1966,doCarmo92} for an introduction to Riemannian manifolds
    and to~\cite{Absil2008} for optimization thereon. We will use the following
    terminology and notations.

    We denote by $T_a\M$ the (Euclidean) tangent space to $\M$ at $a \in \M$;
    $T\M \coloneqq \cup_{a\in\M}T_a\M$ is the tangent bundle to $\M$; $\langle
    \cdot, \cdot \rangle_a$ denotes the inner product in the tangent space $T_a\M$ at
    $a$ and from which we deduce the norm of $v \in T_a\M$ denoted by $\lVert v
    \rVert_{a} = \sqrt{\langle v, v \rangle_a}$. For a (not necessarily unique)
    shortest geodesic between $a$ and $b\in \M$, we write $g(\cdot;a,b)\colon \R \to \M$,
    $t\mapsto g(t;a,b)$, parametrized such that \(g(0;a,b) = a\)
    and~\(g(1;a,b) = b\). This choice of parametrization also means
    that the covariant derivative $\frac{D}{dt}$ of \(g\) with respect to time
    satisfies~\(\lVert \tfrac{D}{dt}g(t;a,b)\rVert_{g(t)} = d_{\mathcal M}(a,b)\),
    for all $t\in[0,1]$, where $d_\M(a,b)$ is the geodesic distance between $a$ and $b \in \M$.
    The Riemannian exponential reads $\Exp{a}\colon T_a\M \to \M, v \mapsto b = \Exp{a}[v]$
    and we denote by \(r_a\in\mathbb R\) the maximal radius such
    that the exponential map is bijective on \(\mathcal D_a \coloneqq\{ b\in \M :
    d_{\mathcal M}(a,b) < r_a\}\). Then $\Log{a}\colon \mathcal D_a
    \to T_a\M, b \mapsto v = \Log{a}[b]$ is called the Riemannian logarithm
    which is (locally) the inverse of the exponential. A Riemannian manifold is
    called symmetric in \(x\in\mathcal M\) if the geodesic reflection \(s_x\) at
    \(x\in\mathcal M\) given by the mapping \(\gamma(t)\mapsto\gamma(-t)\)
    is an isometry at least locally near \(x\),
    for all geodesics through \(\gamma(0)=x\). If \(\mathcal M\) is symmetric
    in every \(x\in\mathcal M\), the
    manifold is called (Riemannian) symmetric space or symmetric manifold.

    In the following we assume, that both the exponential and the logarithmic
    map are available for the manifold and that they are computationally not too
    expensive to evaluate. Furthermore, we assume that the manifold is
    symmetric.
  %
  % Composite Bézier curves on manifolds
  \subsection{Composite Bézier curves on manifolds}
    One well-known way to generalize Bézier curves to a Riemannian manifold
    $\mathcal{M}$ is via the De Casteljau algorithm~\cite[Sect.~2]{Popiel2007}.
    This algorithm only
    requires the Riemannian exponential and logarithm and conserves the
    interpolation property of the first and last control points. Some examples
    on interpolation and fitting with Bézier curves or Bézier surfaces,
    \ie, generalizations of tensor product Bézier curves,
    can be found in Absil~\etal~\cite{Absil2016}.

    Consider $\beta_{K}\colon[0,1] \to \M,
    t \mapsto \beta_{K}(t;b_0,b_1,\dots,b_K)$,
    the manifold-valued Bézier curve of order
    $K$ driven by $K+1$ control points $b_0,\ldots,b_K\in\M$.
    We introduce the points
    $x_i^\bc{0} = b_i$ and iterate the construction of further points. 
    For \(i=0,\ldots,K-k\), \(k=1,\ldots,K\), we define
    \begin{equation}\label{eq:castel}
      x_i^\bc{k} \coloneqq \beta_k(t;b_i,\dots,b_{i+k})
      = g(t;x_i^\bc{k-1},x_{i+1}^\bc{k-1})
    \end{equation}
    as the $i^\text{th}$ point of the $k^\text{th}$ step of the De Casteljau
    algorithm, and obtain $\beta_K(t;b_0,\dots,b_K) = x_0^\bc K$.

    The De Casteljau algorithm is illustrated on Fig.~\ref{fig:casteljau} for
    a Euclidean cubic Bézier curve $\beta_3(t;b_0,b_1,b_2,b_3)$. The general
    cubic Bézier curve can be explicitly expressed on a manifold $\M$ as
    \begin{align*}
      \beta_3(t;x^\bc 0_0,x^\bc 0_1,x^\bc 0_2,x^\bc 0_3)
        &= g
        \Bigl(t;
          g\bigl(t;
            g(t;x^\bc 0_0,x^\bc 0_1)
            ,
            g(t;x^\bc 0_1,x^\bc 0_2)
          \bigr)
          , \\
          & \qquad\quad\
          g\bigl(t;
            g(t;x^\bc 0_1,x^\bc 0_2)
            ,
            g(t;x^\bc 0_2,x^\bc 0_3)
          \bigr)
        \Bigl)
        \\
        &= g\bigl(t;
          g(t;x^\bc 1_0,x^\bc 1_1),
          g(t;x^\bc 1_1,x^\bc 1_2)
        \bigr)
        \\&= g(t; x^\bc 2_0,x^\bc 2_1)
        \\&=x_0^\bc 3.
    \end{align*}
    \begin{figure}[tbp]
      \centering
        \includegraphics{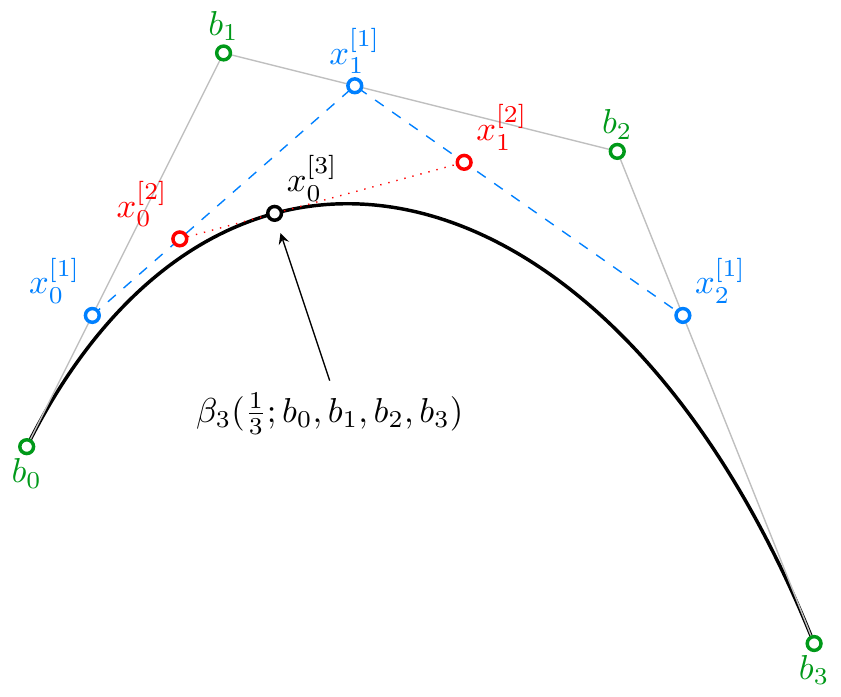}
      \caption{Construction of a cubic Bézier curve via the De Casteljau
      algorithm.}
      \label{fig:casteljau}
    \end{figure}

    The conditions of continuity and differentiability are generalized
    to manifolds in~\cite{Popiel2007}.
    \begin{lem}[Differentiability conditions~\cite{Popiel2007}]
      \label{lem:c0c1-manifold}
      Consider the composite Bézier
      \\curve~\(\bspline \colon [0,2] \to \mathcal M\)
      consisting of
      \(f \colon [0,1] \to \M, t \mapsto
          f(t) = \beta_K(t;b_0^0,b_1^0,\dots,b_K^0)\) and
      \(g \colon [1,2] \to \M, t \mapsto
          g(t) = \beta_{\bar K}(t-1;b_0^1,b_1^1,\dots,b_{\bar K}^1)\), i.e.
      \[
        \bspline(t) \coloneqq \begin{cases}
          f(t) & \text{ for } t \in [0,1], \\
          g(t) & \text{ for } t \in (1,2].
        \end{cases}
      \]
      The composite Bézier curve~$\bspline(t)$ is continuous and
      continuously differentiable if the two following conditions hold:
      \begin{align}\label{eq:c1cond}
        b_K^0 = b_0^1
        \qquad\text{and}\qquad
        b_K^0 &= g(s; b_{K-1}^0, b_1^1),
        \quad s = \frac{K}{K + \bar{K}}.
      \end{align}
    \end{lem}

    The composite Bézier curve \(\bspline\colon[0,n] \to \mathcal M\) is then
    defined completely analogously to the Euclidean case from
    Eq.~\eqref{eq:RnCompBezier}. The differentiability
    conditions~\eqref{eq:c1cond} of Lemma~\ref{lem:c0c1-manifold} have to be
    satisfied at each junction point~$p_{i}$,
    $i=1,\ldots,n-1$.

    \begin{figure}[tbp]\centering
        \includegraphics[width=0.42275\textwidth]{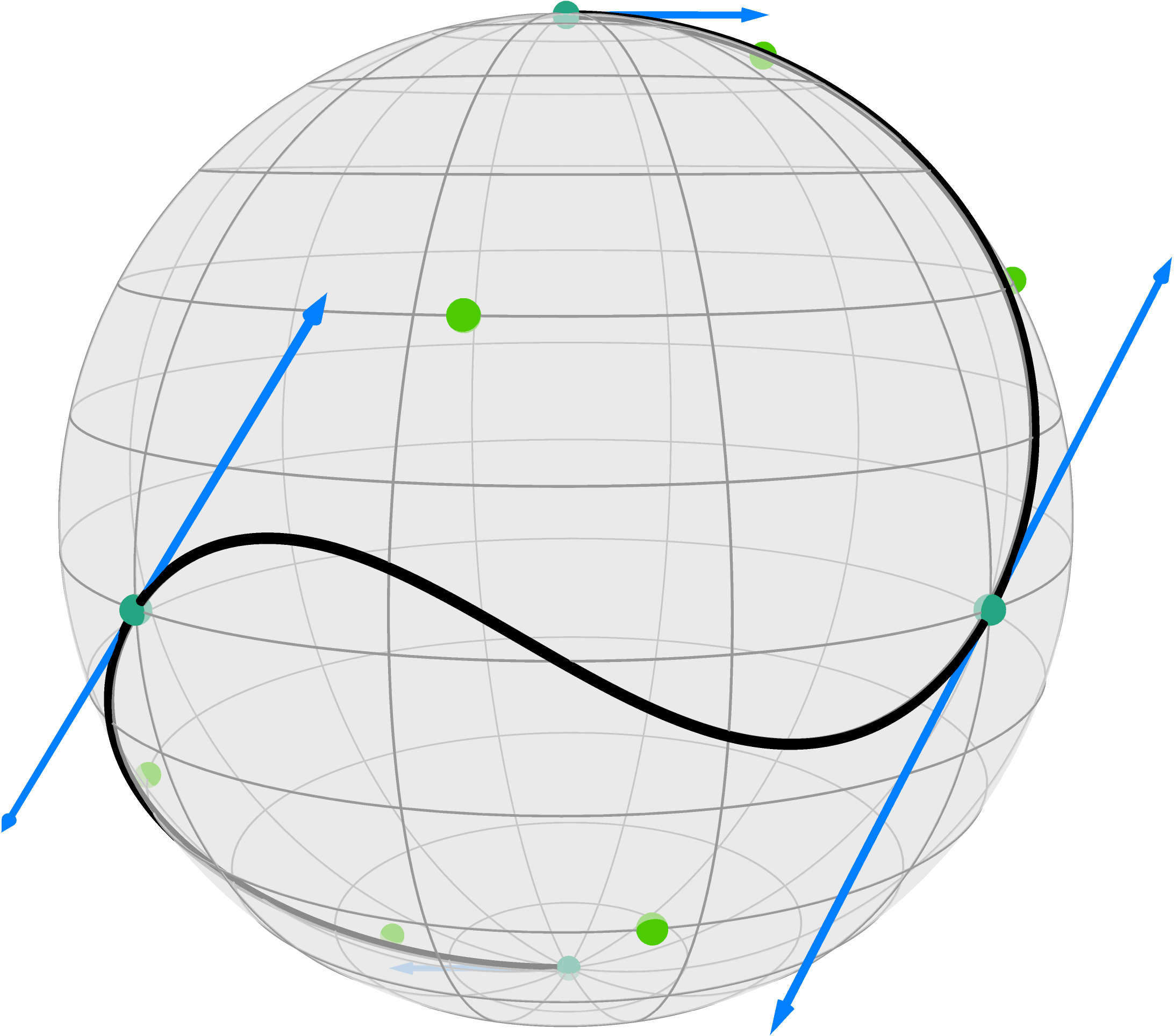}
        \caption{%
       A composite cubic Bézier curve~\(\bspline\colon [0,3] \to \mathbb S^2\).
       The end points $p_i$, $i=0,\ldots,3$, (cyan) and intermediate
       points~$b_i^{\pm}$ (green) determine its shape; continuous
       differentiability is illustrated by the logarithmic
       maps~$\log_{p_i}b_i^{\pm}$, $i\in \{1,2\}$ (blue arrows).
    }
        \label{fig:bezierS2}
    \end{figure}
    \begin{exmp}[Composite cubic Bézier curves]
      The composite cubic Bézier\\
      curve \( \bspline(t) \colon [0,n] \to \M\)
      is $C^1$ if \(p_i = g(\frac{1}{2}; b_i^-, b_i^+)\).
      See Fig.~\ref{fig:bezier} for an example on $\M = \R$ with $n=5$
      segments
      and Fig.~\ref{fig:bezierS2} for an example on $\M=\mathbb S^2$
      with $n=3$ segments.
    \end{exmp}
  %
  % Optimal composite Bezier curve
  \subsection{Discrete approximation of the mean squared acceleration}
    \label{subsec:opt_bezier}
    We discretize the mean squared acceleration (MSA) of a
    curve~$\gamma\colon [0,1]\to\M$, by discretizing the corresponding integral
    \begin{equation}
        \int_0^1 \Bigl\lVert \frac{D^2\gamma(t)}{\d t^2} \Bigr\rVert^2_{\gamma(t)} \d t,
    \end{equation}
    \ie, the regularizer from~\eqref{eq:E}.
    We approximate the squared norm of the
    second (covariant) derivative by the second order absolute finite difference
    introduced by Ba\v{c}\'ak~\etal~\cite{Bacak2016}.
    Consider three points $x,y,z \in \M$ and the \defTerm{set of mid-points of $x$ and $z$}
    \[
      \mathcal C_{x,z}\coloneqq
      \bigl\{
        c
        \,\big|\,
        c = g(\tfrac12;x,z)
        \bigr\},
    \]
    for all (not necessarily shorest) geodesics $g$ connecting $x$ and $z$.
    The \defTerm{manifold-valued second order absolute finite differences}
    is defined by
    \begin{align}\label{eq:secondOrderDifference}
      \d_2[x,y,z] = \min_{c\in\mathcal C_{x,z}}2d_\M(c,y).
    \end{align}
    This definition is equivalent, on the Euclidean space, to
    $\lVert x - 2y + z\rVert = 2\lVert \frac{1}{2}(x+z)-y\rVert$.

    Using equispaced points $t_0,\ldots,t_N$, $N\in\mathbb N$, with
    step size $\Delta_t = t_1-t_0 = \frac{1}{N}$,
    we approximate $\Bigl\lVert\frac{D^2\gamma(t_i)}{\d t^2}\Bigr\rVert_{\gamma(t_i)}
        \approx \frac{1}{\Delta_t^2}d_2[\gamma(t_{i-1}),\gamma(t_{i}),\gamma(t_{i+1})]$,
        $i=1,\ldots,N-1$, and obtain by the trapezoidal rule
    \[
    \int_0^1 \Bigl\lVert \frac{D^2\gamma(t)}{\d t^2} \Bigr\rVert^2_{\gamma(t)}\d t
    \approx
    \sum_{k=1}^{N-1}
    \frac{ \Delta_t d_2^2[\gamma(t_{i-1}),\gamma(t_i),\gamma(t_{i+1})]}{\Delta_t^4}.
    \]
    For Bézier curves $\gamma(t)=\bspline(t)$ we obtain for the regularizer
    in~\eqref{eq:E} the \defTerm{discretized MSA} \(A(\vect{b})\) that
    depends on the control points $\vect{b}$ and reads
    \begin{equation}
      \label{eq:discr_obj}
      A(\vect{b}) \coloneqq
      \sum_{i=1}^{N-1}
        \frac{\d^2_2 \left[
            \bspline(t_{i-1}),\bspline(t_{i}),\bspline(t_{i+1})
          \right]}{\Delta_t^3}.
    \end{equation}
%
% --- Gradient
\section{The gradient of the discretized mean squared acceleration}
\label{sec:Gradient}
  In order to minimize the discretized MSA \(A(\vect{b})\), we aim to employ a
  gradient descent algorithm on the product manifold \(\M^M\), where \(M\) is the number
  of elements in \(\vect{b}\). In the following, we derive a closed form of the gradient
  $\nabla_{\vect{b}} A(\vect{b})$ of the discretized MSA~\eqref{eq:discr_obj}. This gradient
  is obtained by means of a recursion and the chain rule. In fact, the 
  derivative of~\eqref{eq:secondOrderDifference} is already known~\cite{Bacak2016},
  such that it only remains to compute the derivative of the composite B\'ezier curve.

  We first introduce the following notation.
  We denote by \(D_xf[\eta](x_0)\in T_{f(x_0)\M}\) the
  directional derivative of \(f\colon\mathcal M \to\mathcal M\) evaluated at \(x_0\),
  with respect to its argument \(x\) and in the direction \(\eta\in T_x\M\).
  We use the short hand \(D_xf[\eta] = D_xf[\eta](x)\) whenever this
  directional derivative is evaluated afterwards again
  at~\(x\).
  
  We now state the two following definitions, which are crucial for the rest of this section.

  \begin{defn}[{Gradient~\cite[Eq.~(3.31), p.~46]{Absil2008}}]
    \label{def:gradient}
    Let $f\colon \M \to \R$ be a real-valued function on a manifold $\M$,
    $x \in \M$ and $\eta \in T_x \M$.

    The \emph{gradient} $\nabla_\M f(x)\in T_x\mathcal M$ of $f$ at $x$
    is defined as the tangent vector that fulfills
    \begin{equation}\label{eq:defgrad}
      \langle \nabla_{\M}f(x),\eta \rangle_x = D_x f[\eta]
      \quad\text{for all}\quad \eta\in T_x\mathcal M.
    \end{equation}
  \end{defn}
  For multivariate functions~\(f(x,y)\), we denote the gradient of
  \(f\) with respect to \(x\) at \((x_0,y_0)\) by writing
  \(\nabla_{\M,x}f(x_0,y_0)\). We shorten this notation as \(\nabla_{\M,x}f =
  \nabla_{\mathcal M,x} f(x,y)\) when this gradient is seen as a function
  of \(x\) (and \(y\)).
  \begin{defn}[{Chain rule on manifolds~\cite[p.~195]{Absil2008}}] \label{def:chain}
    Let $f\colon \M \to \M$, $h\colon\M \to \M$ be two functions on a manifold $\M$ and $F\colon\M \to \M\),
    \(x\mapsto F(x) = (f \circ h)(x) = f(h(x))$,
    their composition.
    Let $x \in \M$ and $\eta \in T_x \M$.
    The directional derivative \(D_xF[\eta]\) of $F$ with respect to $x$
    in the direction $\eta$ is given by
    \begin{equation}\label{eq:chainrule}
      D_x F[\eta] = D_{h(x)} f\bigl[D_x h[\eta]\bigr],
    \end{equation}
    where $D_x h[\eta] \in T_{h(x)}\M$ and~$D_x F[\eta] \in T_{F(x)}\M$.
  \end{defn}

  The remainder of this section is organized in four parts. We first recall the
  theory on Jacobi fields in Section~\ref{subsec:Jacobi} and their relation to
  the differential of geodesics (with respect to start and end point). 
  In Section~\ref{subsec:coupledGeodesics}, we
  apply the chain rule to the composition of two geodesics, which appears
  within the De Casteljau algorithm. We use this result to build an algorithmic derivation of the differential of
  a general Bézier curve on manifolds with respect to its control points 
  (Section~\ref{subsec:derive-bezier}). We extend the result to composite Bézier
  curves in Section~\ref{subsec:coupledBezier}, including their constraints on
  junction points \(p_i\) to enforce the \(C^1\) condition~\eqref{eq:c1cond}, and finally
  gather these results
  to state the gradient~\(\nabla_{\mathcal M} A(\vect{b})\) of the discretized
  MSA~\eqref{eq:discr_obj} with respect to the control points. 
  
  % Derivative of a gerodesic
  \subsection{Jacobi fields as derivative of a geodesic}
  \label{subsec:Jacobi}
    In the following, we introduce a closed form of the
    differential~\(D_x g(t;\cdot,y)\) of a
    \\geodesic~\(g(t;x,y)\), \(t\in[0,1]\),
    with respect to its start point \(x\in\M\). The differential with respect to
    the end point \(y\in\M\) can be obtained by taking the
    geodesic~\(g(t,y,x) = g(1-t; x,y)\).

    \begin{figure}[tbp]
      \centering
      \includegraphics{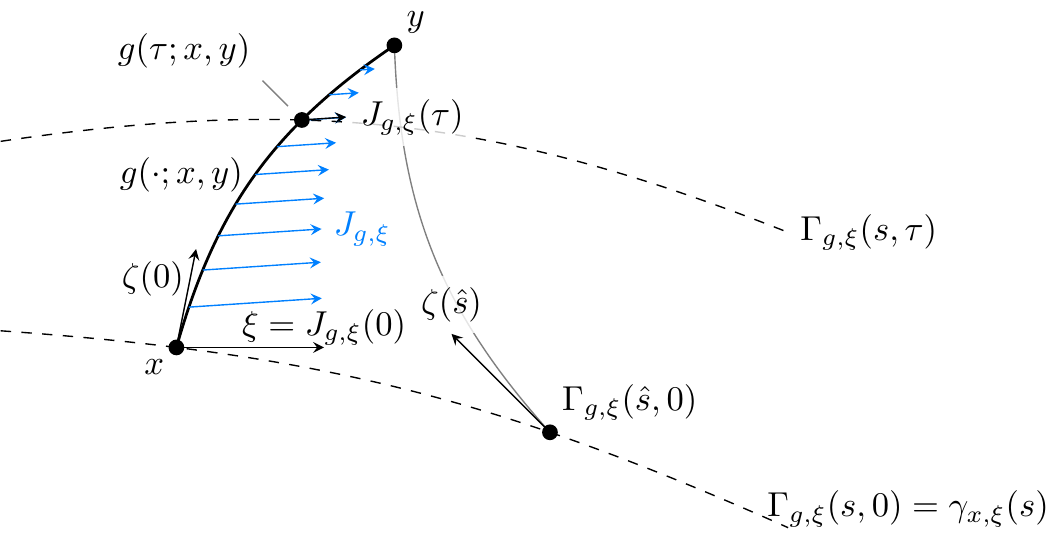}
      \caption{Schematic representation of the variation
        $\Gamma_{g,\xi}(s,t)$ of a geodesic $g$
        w.r.t.~the direction $\xi \in T_x\M$. The
        corresponding Jacobi field along $g$ and in the
        direction $\xi$
        is the vector field
        $J_{g,\xi}(t)
          =
          \frac{\partial}{\partial s}
          \Gamma_{g,\xi}(s,t)\rvert_{s=0}$.
         }
         \label{fig:jacobi}
    \end{figure}

    As represented in Fig.~\ref{fig:jacobi}, we denote by \(\gamma_{x,\xi}\),
    the geodesic starting in \(\gamma_{x,\xi}(0)=x\) and with
    direction~\(\frac{D}{dt}\gamma_{x,\xi}(0) =\xi\in T_x\M\). We introduce
    \(\zeta(s)\coloneqq \Log{\gamma_{x,\xi}(s)}{y}\), the tangential vector in
    \(T_{\gamma_{x,\xi}(s)}\M\) pointing towards \(y\). Then, the
    \emph{geodesic variation} \(\Gamma_{g,\xi}(s,t)\) of the geodesic
    \(g(\cdot;x,y)\) with respect to the tangential direction \(\xi\in T_x\M\)
    is given by
    \[
       \Gamma_{g,\xi}(s,t) \coloneqq \Exp{\gamma_{x,\xi}(s)}[t\zeta(s)],
       \qquad s\in(-\varepsilon,\varepsilon),\ t\in[0,1],
    \]
    where \(\varepsilon > 0\). The corresponding Jacobi field~\(J_{g,\xi}\)
    along~\(g\) is then given by the vector field
    \[
      J_{g,\xi}(t)
      \coloneqq
      \frac{D}{\partial s}\Gamma_{g,\xi}(s,t)\Bigl\rvert_{s=0}
    \]
    that represents the direction of the displacement of
    $g$ if $x$ is perturbed in a direction $\xi$.
    
    We directly obtain~\(J_{g,\xi}(0) = \xi\), and \(J_{g,\xi}(1) =
    0\) as well as~\(J_{g,\xi}(t)\in T_{g(t;x,y)}\M\). Since
    \(\Gamma_{g,\xi}(s,t) = g(t; \gamma_{x,\xi}(s), y)\) we obtain by the chain
    rule
    \begin{equation}
    \label{eq:JacobiDeriv}
      D_x g(t,\cdot,y)[\xi]
      = \frac{d}{ds}\, g(t; \gamma_{x,\xi}(s), y) \Bigr\rvert_{s=0}
      = \frac{d}{ds}\Gamma_{g,\xi}(s,t)\Bigr\rvert_{s=0} = J_{g,\xi}(t).
    \end{equation}
    
    \begin{remark}
    This relation between the derivative of geodesics and
    Jacobi fields is of high interest on symmetric spaces, where Jacobi fields
    can be computed in closed form, as summarized in the following Lemma.
  \end{remark}
    \begin{lem}{\cite[Prop.~3.5]{Bacak2016}}
      \label{lem:symspace}
      Let $\M$ be a $m$-dimensional Riemannian manifold.
      Let \(g(t; x,y)\), \(t\in [0,1]\), be a geodesic between $x,y\in\M$, \(\eta\in T_x\M\) a
      tangent vector and \(\{\xi_1,\ldots\xi_m\}\) be an orthonormal basis
      (ONB) of \(T_x\M\) that diagonalizes the curvature operator of $\M$ with eigenvalues
      $\kappa_\ell$, $\ell=1,\dots,m$. For details, see Ch.~4.2 and~5 (Ex.~5) of~\cite{doCarmo92}.
      Let further denote by~\(\{\Xi_1(t),\ldots,\Xi_m(t)\}\) the parallel transported frame
      of $\{\xi_1,\dots,\xi_m\}$ along $g$. 
      Decomposing~$\eta = \sum_{\ell=1}^m \eta_\ell\xi_\ell \in T_x \M$, 
      the derivative \(D_x g[\eta]\) becomes
      \[
      D_x g(t;x,y) [\eta] = J_{g,\eta}(t) = \sum_{\ell=1}^m \eta_\ell J_{g,\xi_\ell}(t),
      \]
      where the Jacobi field~\(J_{g,\xi_\ell}\colon\mathbb R \to T_{g(t;x,y)}\M\) along \(g\) and in the
      direction~\(\xi_\ell\) is given by
      \begin{equation}\label{eq:DxGeoJacobi}
        J_{g,\xi_\ell}(t) =
        \begin{cases}
          \frac{\operatorname{sinh}\bigl(d_g(1-t)\sqrt{-\kappa_\ell}\bigr)
          }{%
            \sinh(d_g\sqrt{-\kappa_\ell})
          }
          \Xi_\ell(t)&\mbox{if }\kappa_\ell<0,\\
          \frac{  \sin\bigl(d_g(1-t)\sqrt{\kappa_\ell}\bigr) }%
          {  \sin(\sqrt{\kappa_\ell}d_g)}
          \Xi_\ell(t)&\mbox{if }\kappa_\ell>0,\\
          1-t\Xi_\ell(t)&\mbox{if }\kappa_\ell=0,
        \end{cases}
      \end{equation}
      with \(d_g = d(x,y)\) denoting the length of the
      geodesic~\(g(t;x,y), t\in[0,1]\).
    \end{lem}

    The Jacobi field of the reversed geodesic \(\bar
    g(t) \coloneqq g(t;y,x) = g(1-t;x,y)\) is obtained using the same
    orthonormal basis and transported frame
    but evaluated at \(s=1-t\). We thus obtain \(D_y
    g(t;x,y)[\xi_\ell] = D_y g(1-t;y,x)[\xi_\ell] = J_{\bar g,\xi_\ell}(1-t)\), where
    \[
      J_{\bar g,\xi_\ell}(1-t)
      =
      \begin{cases}
        \frac{\operatorname{sinh}\bigl(d_gt\sqrt{-\kappa_\ell}\bigr)
        }{%
          \sinh(d_g\sqrt{-\kappa_\ell})
        }
        \Xi_\ell(t)&\mbox{if }\kappa_\ell<0,\\
        \frac{  \sin\bigl(d_gt\sqrt{\kappa_\ell}\bigr) }%
        {  \sin(d_g\sqrt{\kappa_\ell})}
        \Xi_\ell(t)&\mbox{if }\kappa_\ell>0,\\
        t\Xi_\ell(t)&\mbox{if }\kappa_\ell=0.
      \end{cases}
    \]
  %
  % Derivative of a coupled Geodesic
  \subsection{Derivative of coupled geodesics}
  \label{subsec:coupledGeodesics}
    Let $\M$ be a symmetric Riemannian manifold.
    We use the result of Lemma~\ref{lem:symspace} to directly compute the
    derivative of coupled geodesics, \ie, a function composed
    of~\(g_1(t) \coloneqq g(t;x,y)\) and \(g_2(t) \coloneqq g(t;g_1(t),z)\).
    By Definition~\ref{def:chain}, we have
    \[
      D_x g_2(t)[\eta]
      = D_{g_1(t)}g(t;\cdot,z)\bigl[D_xg_1(t)[\eta]\bigr]
    \]
    and by~\eqref{eq:JacobiDeriv}, we obtain
    \[
      D_x g_2(t)[\eta]
      = J_{g_2,D_xg_1(t)[\eta]}(t),
    \]
    where the variation direction used in the Jacobi field is now the
    derivative of $g_1(t)$ in direction $\eta$.
    Similarily, we compute the derivative of a reversed coupled
    geodesic~$g_3(t) \coloneqq g(t;z,g_1(t))$ as
    \[
      D_x g_3(t)[\eta]
      = D_{g_1(t)}g(t;z,\cdot)\bigl[D_xg_1(t)[\eta]\bigr]
      = J_{\bar g_3,D_xg_1(t)[\eta]}(1-t).
    \]
    Note that the Jacobi field is here reversed, but that its variation
    direction is the same as the one of the Jacobi field introduced for
    $g_2(t)$. In a computational perspective, it means that we can use the same ONB
    for the derivatives of both~\(g_3\) and \(\bar g_3\)
    Furthermore, in this case, the variation direction is also computed by
    a Jacobi field since~$D_xg_1(t)[\eta] = J_{g_1,\eta}(t)$.

    Finally the derivative of $g_2$ (resp.~$g_3$) on symmetric spaces
    is obtained as follows. Let
    $\{ \xi_1^{[1]},\dots,\xi_m^{[1]} \}$ be an ONB of $T_x\M$ for the
    inner Jacobi field along \(g_1\), and $\{
    \xi_1^{[2]},\dots,\xi_m^{[2]} \}$ be an ONB
    of $T_{g_1(t)}\M$ for the outer Jacobi
    field along $g_2$ (resp. $g_3$). As $\eta = \sum_{\ell=1}^m \eta_\ell
    \xi^{[1]}_\ell \in T_x\M$, and stating $J_{g_1,\xi^{[1]}_\ell}(t) = \sum_{l=1}^m
    \mu_l \xi_l^{[2]} \in T_{g_1(t)}\M$, the derivative of $g_2$ (resp. $g_3$)
    with respect to $x$ in the direction $\eta \in T_x\M$ reads
    \begin{equation}
    \label{eq:coupledDeriv}
    D_x g_2(t)[\eta]
    = \sum_{l=1}^m \sum_{\ell=1}^m J_{g_2,\xi^{[2]}_l}(t) \mu_l \eta_\ell,
    \end{equation}
    and accordingly for $g_3$.
  %
  % Derivative of a bezier curve
  \subsection{Derivative of a Bézier curve}\label{subsec:derive-bezier}
    Sections~\ref{subsec:Jacobi} and~\ref{subsec:coupledGeodesics} introduced
    the necessary concepts to compute the derivative of a general Bézier curve
    \(\beta_K(t;b_0,\dots,b_K)\), as described in Equation~\eqref{eq:castel},
    with respect to its control points $b_j$.
    For readability of the recursive structure investigated in the following,
    we introduce a slightly simpler notation and the following setting.
    
    Let $K$ be the degree of a Bézier curve $\beta_K(t;b_0,\dots,b_K)$ with
    the control points $b_0,\dots,b_K\in\M$. We fix \(k\in \{1,\ldots,K\}\),
    \(i\in\{0,\ldots,K-k\}\) and \(t\in[0,1]\). We introduce
    \begin{equation}
      \label{eq:short_bezier_notation}
      g_i^\bc k (t) \coloneqq g(t; g_i^\bc {k-1}(t), g_{i+1}^\bc {k-1}(t))
        = \beta_i^\bc k (t; b_i,\dots,b_{i+k}),
    \end{equation}
    for the \(i^\text{th}\) Bézier curve of degree \(k\) in the De Casteljau
    algorithm, and \(g_i^\bc{0}(t) = b_i\). 
    
    Furthermore, given $x \in \{ b_i,\dots,b_{i+k} \}$, we denote by
    \begin{equation}
    \label{eq:etaNotation}
    \eta_i^\bc k \coloneqq D_x g_i^\bc k(t)[\eta],
    \end{equation}
    its derivative with respect to one of its control points~\(x\)
    in the direction \(\eta \in T_x\M\). 
    
    \begin{remark}
      Clearly any other derivative of $g_i^\bc{k}$ with respect
      to \(x = b_j\), \(j<i\)  or \(j>i+k\) is zero. In addition we have
    \(\eta_i^\bc{0}=D_xg_i^\bc{0}[\eta] = \eta\) for \(x=b_i\) and zero otherwise.
    \end{remark}
    \begin{thm}[Derivative of a Bézier curve]
      \label{thm:bezier_derivative}
      Let \(k\in \{1,\ldots,K\}\) and \(i\in\{0,\ldots,K-k\}\) be given.
      The derivative \(\eta_i^\bc{k} = D_x g_i^\bc k(t)[\eta]\) of \(g_i^\bc k\)
      with respect to its control point \(x \coloneqq b_j\), \(i \leq j \leq i+k\),
      and in the direction \(\eta \in T_x \M\) is given by
      \begin{equation*}
        \eta_i^\bc k \coloneqq D_x g_i^\bc k (t) [\eta]
        =
        \begin{cases}
          J_{g_i^\bc{k-1},\eta_i^\bc{k-1}}(t)
            &\text{ if } j = i,\\
          J_{g_i^\bc{k-1},\eta_i^\bc{k-1}}(t)
            + J_{\bar g_{i+1}^\bc{k-1},\eta_{i+1}^\bc{k-1}}(1-t)
            &\text{ if } i<j<i+k,\\
          J_{\bar g_{i+1}^\bc{k-1},\eta_{i+1}^\bc{k-1}}(1-t)
            &\text{ if } j = i+k.
        \end{cases}
      \end{equation*}
    \end{thm}
    \begin{proof}
      Let fix \(t\in[0,1]\)
      and $x=b_j$, $i\leq j \leq i+k$. For readability we
      set \(a \coloneqq g_i^\bc {k-1}(t)\), \(b \coloneqq g_{i+1}^\bc
      {k-1}(t)\), and \(f \coloneqq g_i^\bc k(t) = g(t;a,b)\). Note that while
      \(f\) depends on the control points \(b_i,\ldots,b_{i+k}\) and is a
      Bézier curve of degree \(k\), both \(a\) and \(b\) are Bézier curves of
      degree \(k-1\). The former does not depend
      on~$b_{i+k}$, and the latter is independent of \(b_{i}\).

      We prove the claim by induction. For \(k=1\) the function \(g_i^\bc{1}\)
      is just a geodesic. The case \(i<j<i+1\) does not occur and the remaining
      first and third cases follow by the notation introduced for \(k=0\) and
      Lemma~\ref{lem:symspace}.

      For \(k>1\) we apply the chain rule~\eqref{eq:chainrule} to~\(D_xf[\eta]\)
      and obtain
      \[
        D_x f[\eta]
        = D_a f \bigl[ D_x a [\eta]\bigr]
        +
        D_b f \bigl[ D_x b [\eta]\bigr].
      \]
      Consider the first term \(D_af\bigl[D_xa[\eta]\bigr]\) and $j < i+k$. 
      By~\eqref{eq:JacobiDeriv} and the notation
      from~\eqref{eq:etaNotation}, one directly has
      \[
        D_a f\bigl[\eta_i^\bc{k-1}\bigr]
        = J_{f, \eta_i^\bc{k-1}}(t).
      \]
      For \(j=i+k\), clearly~\(D_xa[\eta]=D_af[D_xa[\eta]]=0\), as $a$ does not
      depend on~$b_{i+k}$.

      We proof the second term similarily. For \(j>i\), by applying the chain
      rule and using the reversed Jacobi field formulation of
      Lemma~\ref{lem:symspace} for the derivative of a geodesic with respect to its end point, we obtain
      \[
        D_{b} f\bigl[\eta_{i+1}^\bc{k-1}\bigr]
        = J_{\bar f, \eta_{i+1}^\bc{k-1}}(1-t).
      \]
      Finally, as~\(D_xb[\eta]=D_bf[D_xb[\eta]]=0\) for \(x=b_i\), the
      assumption follows.
    \end{proof}
    Fig.~\ref{fig:derivation_tree} represents one level of the
    schematic propagation tree to compute the derivative of a Bézier curve.
    \begin{figure}[tpb]
      \centering
      \begin{subfigure}[t]{.28\textwidth}
          \centering
          \includegraphics{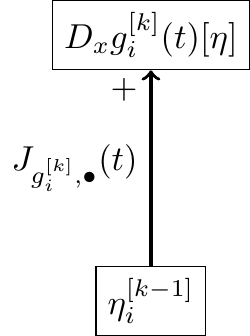}
          \caption{The case $x=b_i$.}
      \end{subfigure}
      \begin{subfigure}[t]{.4\textwidth}
      \centering
          \includegraphics{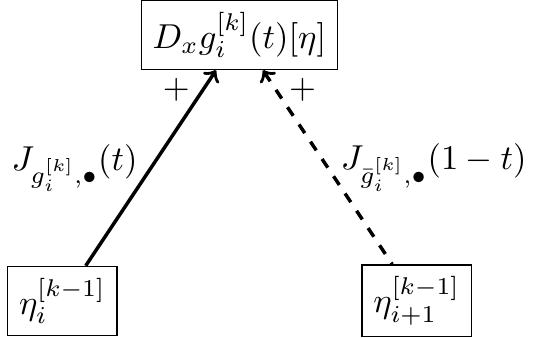}
          \caption{The cases $x\in\{b_{i+1},\ldots,b_{i+k-1}\}$.}
      \end{subfigure}
      \begin{subfigure}[t]{.28\textwidth}
      \centering
          \includegraphics{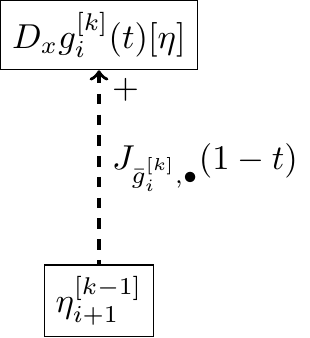}
          \caption{The case $x=b_{i+k}$.}
      \end{subfigure}
      \caption{Schematic representation of the cases where elements compose
        the chained derivative of the \(i^\text{th}\) Bézier curve
        of order \(k\) in the De Casteljau algorithm.
        The solid line represents a Jacobi
        field along \(g_i^\bc k\), while the dashed one represents a
        \emph{reversed} Jacobi field.}
      \label{fig:derivation_tree}
    \end{figure}

    \begin{exmp}[Quadratic Bézier curve]
      \label{ex:quadratic}
      Consider the quadratic Bézier curve \( \beta_2\colon [0,1] \to \M \)
      defined as
      \[
        \beta_2(t; b_0,b_1,b_2)
        =
        g\bigl(
        t;
        g(t; b_0,b_1),
        g(t;b_1,b_2)
        \bigr).
      \]
      Using the notations~\eqref{eq:short_bezier_notation}, we have
      \begin{align*}
        g_0^\bc 1(t) &\coloneqq g(t; b_0, b_1),\quad
        g_1^\bc 1(t) \coloneqq g(t; b_1, b_2), \\
        g_0^\bc 2(t) &\coloneqq g(t;g_0^\bc 1,g_1^\bc 1).
      \end{align*}
      The derivative of \(\beta_2\) at $t$ with respect to \(b_0\) in the
      direction \(\eta \in T_{b_o}\M\), is given by
      \[
        D_{b_0}\beta_2[\eta] = J_{g_0^\bc 2, \eta_0^\bc 1}(t),
        \text{ with }
        \eta_0^\bc 1 \coloneqq J_{g_0^\bc 1, \eta}(t).
      \]
      The derivative of \(\beta_2\) at $t$ with respect to \(b_2\) in the
      direction \(\eta \in T_{b_2}\M\) can be seen as deriving by the first
      point after inverting the Bezier curve, i.e.~looking at \(\bar\beta_2(t)
      = \beta_2(1-t)\), hence we have analogously to the first term
      \[
        D_{b_2}\beta_2[\eta] = J_{\bar g_0^\bc 2, \eta_1^\bc 1}(1-t),
        \text{ with }
        \eta_1^\bc 1 \coloneqq J_{\bar g_1^\bc 1, \eta}(1-t).
      \]
      The case~\(D_{b_1}\beta_2[\eta]\), \(\eta \in T_{b_1}\M\), involves a
      chain rule, where \(b_1\) appears in both~\(g_0^\bc 1\) (as its starting point) and~\(g_1^\bc
      1\) (as its end point). Using the two intermediate results (or Jacobi fields of geodesics)
      \begin{equation*}
        \eta_0^\bc 1  \coloneqq J_{\bar g_0^\bc 1, \eta}(1-t)
        \quad\text{and}\quad
        \eta_1^\bc 1  \coloneqq J_{     g_1^\bc 1, \eta}(  t),
      \end{equation*}
      we obtain
      \[
        D_{b_1}\beta_2[\eta] =
        J_{g_0^\bc 2, \eta_0^\bc 1}(t)
        +
        J_{\bar g_0^\bc 2, \eta_1^\bc 1}(1-t).
      \]
    \end{exmp}

    \begin{exmp}[Cubic Bézier curve]
      Consider the cubic Bézier curve \( \beta_3\colon [0,1] \to \M \) defined
      as
      \[
        \beta_3(t; b_0,b_1,b_2,b_3)
        =
        g\bigl(
          t;
          \beta_2(t; b_0,b_1,b_2),
          \beta_2(t; b_1,b_2,b_3)
        \bigr).
      \]
      As in Example~\ref{ex:quadratic}, we use the
      notations~\eqref{eq:short_bezier_notation} and define
      \begin{align*}
        g_j^\bc 1(t) & \coloneqq g(t; b_j, b_{j+1}), \qquad j = 0,1,2, \\
        g_j^\bc 2(t) & \coloneqq g(t; g_j^\bc 1, g_{j+1}^\bc 1),
          \qquad j = 0,1, \text{ and }\\
        g_0^\bc 3(t) & \coloneqq g(t;g_0^\bc 2,g_1^\bc 2).
      \end{align*}

      The derivation of \(\beta_3\) with respect to \(b_0\) or \(b_3\) follows
      the same structure as in Example~\ref{ex:quadratic}. The case of \(
      D_{b_1}\beta_3[\eta] \), however, requires two chain rules. The needed
      Jacobi fields follow the tree structure shown in Fig.~\ref{fig:der_tree}:
      given \(\eta \in T_{b_1}\M\), we define at the first recursion step
      \begin{equation*}
        \eta_0^\bc 1 \coloneqq  J_{\bar g_0^\bc 1, \eta}(1-t),
        \qquad
        \eta_1^\bc 1 \coloneqq  J_{     g_1^\bc 1, \eta}(  t),
      \end{equation*}
      and at the second recursion step
      \begin{equation*}
        \eta_0^\bc 2 \coloneqq  J_{g_0^\bc 2, \eta_0^\bc 1}(t)
          + J_{\bar g_0^\bc 2, \eta_1^\bc 1}(1-t),
        \qquad
        \eta_1^\bc 2 \coloneqq  J_{g_1^\bc 2, \eta_1^\bc 1}(t).
      \end{equation*}
      Note that both \(\eta_0^\bc 2\) and \(\eta_1^\bc 2\) are actually the derivative of
      \(\beta_2(t;b_0,b_1,b_2)\) and \(\beta_2(t;b_1,b_2,b_3)\), respectively, with respect to \(b_1\) and direction
      \(\eta\in T_{b_1}\M\). Finally we have
      \[
        D_{b_1}\beta_3[\eta] =
        J_{g_0^\bc 3, \eta_0^\bc 2}(t)
        +
        J_{\bar g_0^\bc 3, \eta_1^\bc 2}(1-t).
      \]
      The case of \( D_{b_2}\beta_3[\eta] \) is obtained symmetrically, with
      arguments similar to~\(b_1\).
    \end{exmp}

    \begin{figure}[tbp]
      \centering
      \begin{subfigure}[t]{.45\linewidth}\centering
        \includegraphics{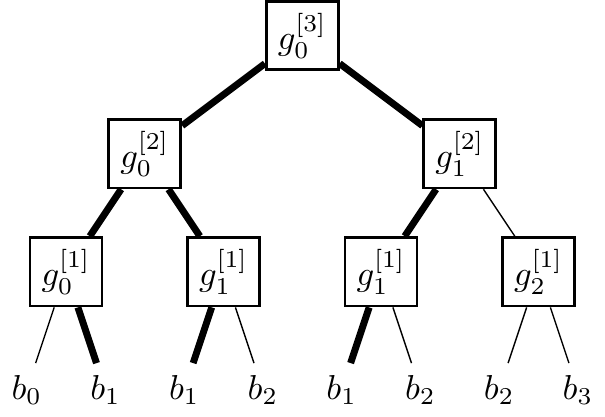}
        \caption{Tree-representation of the
          construction of a cubic Bézier curve. The thick line tracks
          the propagation of \(b_1\) within the tree.}
        \label{fig:constr_tree}
      \end{subfigure}
      \begin{subfigure}[t]{.45\linewidth}\centering
        \includegraphics{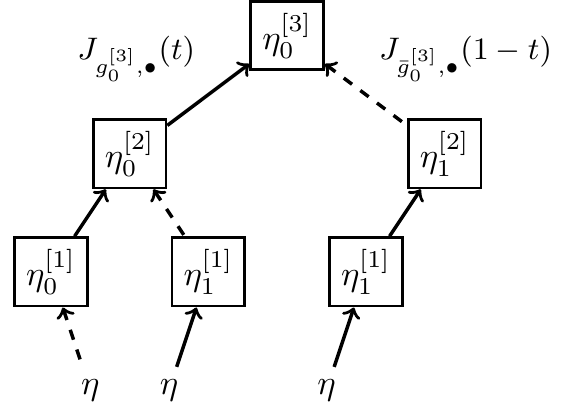}
        \caption{Tree-representation of the
          recursive construction of
          \(\eta_0^\bc 3 \coloneqq D_{b_1}\beta_3[\eta]\).
          The solid lines are Jacobi fields while dashed lines are
          \emph{reversed} Jacobi fields.}
        \label{fig:der_tree}
      \end{subfigure}
      \caption{Construction and derivation tree of a Bézier curve
        \(\beta_3(t;b_0, b_1, b_2, b_3)\).
        The derivative with respect to a variable \(b_i\) is obtained
        by a recursion of Jacobi fields added at each leaf of the tree.
      }
      \label{fig:tree}
    \end{figure}
  %
  % Gradient of the optimal Bezier curve
  \subsection{Joining segments and deriving the gradient}
  \label{subsec:coupledBezier}
    In this subsection we derive the differential of a composite Bézier
    curve~\(\bspline(t)\) consisting of $n$ segments and take the \(C^1\)
    conditions into account. 
    We simplify the notations from
    Section~\ref{subsec:BezierRn} and set the degree fixed to \(K_i=K\) for
    all segments, \eg,~\(K=3\) for a cubic composite Bézier curve. Then the
    control points are~\(b_j^i\), \(j=0,\ldots,K\), \(i=0,\ldots,n-1\). We further
    denote by \(p_i=b_{K}^{i-1}=b_0^i\), \(i=1,\ldots,n-1\) the common junction
    point of the segments and \(p_0\) and \(p_n\) the start and end points,
    respecively. For ease of notation we denote by \(b_i^- = b_{K-1}^i\) and
    \(b_i^+=b_1^i\), $i=1,\ldots,n-1$, the two points needed for differentiability (\(C^1\))
    condition investigation, cf.~Fig.~\ref{fig:bezier} for an illustration of
    the framework on $\M = \R$, with $K=3$.

    One possibility to enforce the \(C^1\) condition~\eqref{eq:c1cond} is to
    include it into the composite Bézier curve by replacing \(b_i^+\) with
    \begin{equation}\label{eq:C1constraint}
      b_i^+ = g(2; b_i^-,p_i),\quad i=1,\ldots,n-1.
    \end{equation}
    This way both the directional derivatives of $\bspline(t)$ with respect
    to~\(b_i^+\) and \(p_i\) change due to a further (most inner) chain rule.
    \begin{lem}[Derivative of a composite Bézier curve with \(C^1\) condition]
    \label{lem:derC1}
      Let \(\bspline\) be a composite Bézier curve and \(p_i,b_i^+,b_i^-\)
      introduced as above. Replacing \(b_i^+= g(2;b_i^-,p_i)\) eliminates that
      variable from the composite Bezier curve and keeps the remaining
      differentials unchanged, despite the following which now read
      \begin{equation*}
        D_{b_i^-}\bspline(t)[\eta] =
        \begin{cases}
      % case t \in (i-1, i]
      D_{b_i^-} \beta_K(t-i+1; p_{i-1},b_{i-1}^+,\ldots,\cdot,p_i)[\eta]
        &t\in (i-1,i],\\
      % case t \in (i,i+1]
      D_{b_i^+} \beta_K(t-i; p_i,b_i^+,\ldots,\cdot,p_{i+1})
              \left[
                D_{b_i^-}g(2; \cdot,p_i)[\eta]
            \right]
            &t\in (i,i+1],
        \end{cases}
      \end{equation*}
      and
      \begin{equation*}
        D_{p_i}\bspline(t)[\eta] =
        \begin{cases}
      % case t \in (i-1,i]
      D_{p_i}  \beta_K(t-i+1; p_{i-1},b_{i-1}^+,\ldots,b_i^-,\cdot)[\eta]
        &t\in (i-1,i],\\
      % case t \in (i,i+1]
      D_{p_i} \beta_K(t-i; \cdot,b_i^+,\ldots,b_{i+1}^-,p_{i+1})[\eta]\\
      \quad+
      D_{b_i^+} \beta_K(t-i; p_i,\cdot,\ldots,b_{i+1}^-,p_{i+1})
              \left[
                D_{p_i}g(2; b_i^-,\cdot)[\eta]
            \right]
        &t\in (i,i+1].
        \end{cases}
      \end{equation*}
      In both cases, the first interval includes $i-1=0$, when $i=1$.
    \end{lem}
    \begin{proof}
      Both first cases are from the derivation of a bezier curve as before,
      for both second cases replacing \(b_i^+ = g(2; b_i^-,p_i)\) yields one
      (for \(p_i\) additional) term.
    \end{proof}
  %
  % Grad of dMSA
    We now derive the gradient of the objective function~\eqref{eq:discr_obj}.
    We introduce the abbreviation \(\bspline_i = \bspline(t_i) \in \M\),
    and \(\d^2_{2,i} = \d^2_2 ( \bspline_{i-1},\bspline_i,\bspline_{i+1} )\).
    \begin{thm}
      \label{lemma:optimal-bezier-curve}
      Let $\M$ be a $m$-dimensional manifold, \( x \) be one of the control
      points of a composite Bézier curve~\(\bspline\), 
      and \(\{ \xi_1,\dots,\xi_m \}\) be a corresponding orthonormal basis (ONB) of \(T_x \M\).
      The gradient $\nabla_{\M,x}A(\vect{b})$ of the discretized mean squared
      acceleration $A(\vect{b})$ of~\(\bspline\), w.r.t. \(x\), discretized
      at $N+1$ equispaced times \(t_0,\ldots,t_N\) is given by
      \[
      \nabla_{\M,x}A(\vect{b}) =
        \sum_{\ell=1}^m
        \sum_{i=1}^{N-1}
          \sum_{j = i-1}^{i+1}
          \langle
            \nabla_\M \d^2_{2,j}, D_x \bspline_j[\xi_\ell]
          \rangle_{\bspline_j} \xi_\ell.
      \]
    \end{thm}
    \begin{proof}
      As $\nabla_{\M,x}A(\vect{b}) \in T_x \M$, we seek for the
      coefficients~$a_\ell\coloneqq a_\ell(x)$ such that
      \begin{equation}
        \label{eq:grad0}
        \nabla_\M f(x) = \sum_{\ell=1}^m a_\ell \xi_\ell.
      \end{equation}
      Therefore, for any tangential vector
      $\eta \coloneqq \sum_{\ell=1}^m \eta_\ell \xi_\ell \in T_{x} \M$,
      we have,
      \begin{equation}
        \label{eq:grad1}
        \langle \nabla_{\M,x}A(\vect{b}),\eta \rangle_{x}  = \sum_{\ell=1}^m a_\ell \eta_\ell.
      \end{equation}
      By definition of $A$ and Definition~\ref{def:gradient} this yields
      \begin{equation}
        \label{eq:grad2}
        \langle \nabla_{\M,x}A(\vect{b}),\eta \rangle_{x}
          = \sum_{i=1}^{N-1} \langle \nabla_\M \d^2_{2,i}(x),\eta \rangle_x
          = \sum_{i=1}^{N-1} D_x \d^2_{2,i} [\eta].
      \end{equation}
      We compute $D_x\d^2_{2,i} [\eta]$ using the chain rule (Definition~\ref{def:chain}) as
      \begin{equation}
        \label{eq:grad3}
        D_x \d^2_{2,i} [\eta] =
          \sum_{j = i-1}^{i+1}D_{\bspline_j} \d^2_{2,j} \big[ D_x \bspline_j[\eta]\big],
      \end{equation}
      which, by Definition~\ref{def:gradient}, again becomes
      \[
        D_{\bspline_j} \d^2_{2,j} \big[ D_x \bspline_j[\eta]\big]
          = \langle \nabla_\M \d^2_{2,j} , D_x \bspline_j[\eta]
          \rangle_{\bspline_j}.
      \]
      The term on the left of the inner product is given
      in~\cite[Sec.~3]{Bacak2016} and the right term is given in Section~\ref{subsec:derive-bezier}. 
      While the former can be computed using  Jacobi fields and a logarithmic map, 
      the latter is the iteratively coupling of Jacobi fields.
      Furthermore, the differential $D_x\bspline_j[\eta]$ can be written as
      \[
        D_x \bspline_j[\eta] = \sum_{\ell=1}^m \eta_\ell D_x \bspline_j [\xi_\ell] \in T_{\bspline_j}\M.
      \]
      Hence, we obtain
      \[
        D_{\bspline_j} \d^2_{2,j} \big[ D_x \bspline_j[\eta]\big]
        = \sum_{\ell=1}^m \eta_\ell
        \langle
          \nabla_\M \d^2_{2,j},
          D_x \bspline_j[\xi_\ell]
        \rangle_{\bspline_j},
      \]
      and by~\eqref{eq:grad1}, \eqref{eq:grad2} and \eqref{eq:grad3}, it follows
      \[
      \langle \nabla_{\M,x}A(\vect{b}),\eta \rangle_{x} =
            \sum_{\ell=1}^m \eta_\ell
            \sum_{i=1}^{N-1}
            \sum_{j = i-1}^{i+1}
            \langle \nabla_\M \d^2_{2,j},
              D_x \bspline_j[\xi_\ell] \rangle_{\bspline_j},
      \]
      which yields the assertion~\eqref{eq:grad0}.
    \end{proof}
  % Numerical Algorithms
% ----
\section{Application to the fitting problem}
\label{sec:Numerics}
  The fitting problem has been tackled different
  ways this last decades. The approach with B\'ezier curves is more recent, and
  we refer to~\cite{Arnould2015,Absil2016,Gousenbourger2018} for a detailed 
  overview of these methods.
  
  In this section, we present the numerical framework we use in order to
  fit a composite B\'ezier curve $\bspline$ to a set of data points 
  $d_0,\dots,d_n \in \M$ associated with time-parameters $t_0,\dots,t_n$, 
  such that we meet~\eqref{eq:E}. 
  For the sake of simplicity, we limit the study to the case where $t_i = i$,
  $i=0,\dots,n$. Therefore, the fitting problem~\eqref{eq:E} becomes
  \begin{equation}
  \label{eq:EB}
  \min_{\vect{b} \in \Gamma_\bspline} E_\lambda(\vect{b}) \coloneqq
    \int_{t_0}^{t_n} \Bigl\lVert
      \frac{\mathrm D^2 \bspline(t)}{\d t^2}
    \Bigr\rVert^2_{\bspline(t)}
    \d t
    +
    \frac{\lambda}{2} \sum_{i=0}^n \d^2(p_i,d_i),
  \end{equation}
  where $\Gamma_\bspline \in \M^M$ is the set of the $M$ control points of $\bspline$.
  Remark that, compared to~\eqref{eq:E}, the optimization is now done on the
  product manifold $\M^M$. Furthermore, the fitting term now implies a distance
  between $d_i$ and $p_i$, as $\bspline(t_i) = p_i$.
  
  The section is divided in three parts: the product manifold $\M^M$ is
  defined in Section~\ref{subsec:IPC1}, where the contribution of the
  fitting term in the gradient of $E$ is also presented. 
  Then, we propose an efficient algorithm to compute the gradient of
  the discretized MSA, based on so-called \emph{adjoint Jacobi fields}.
  We finally shortly mention the gradient descent algorithm we use as well as
  the involved Armijo rule.
  
  %
  % IP & Approx
  \subsection{Fitting and interpolation}\label{subsec:IPC1}
    Let us clarify the set $\Gamma_\bspline \in \M^M$ from~\eqref{eq:EB}. 
    We will explicitly present the
    vector \(\vect{b}\) and state its size $M$.
    The set $\Gamma_\bspline$ is the set of the $M$ remaining free control
    points to optimize,
    when the $\C^1$ continuity constraints are imposed. We distinguish two cases:
    (i) the fitting case, that corresponds to formulae presented in
    Section~\ref{sec:Gradient}, and (ii) the interpolation case
    ($\lambda \to \infty$) where the constraint $d_i = p_i$ is imposed as well.
    
    For a given composite Bézier curve \(\bspline\colon[0,n]\to\M\)
    consisting of a Bézier curve of degree \(K\) on each
    segment, and given the \(C^1\) conditions~\eqref{eq:c1cond}, the segments
    are determined by the points
    \begin{equation}
    \label{eq:vecBDetail}
    \vect{b} = (p_0,b_0^+,b_2^0,\ldots,b_{K-2}^0,b_1^-,p_1,\ldots,b_{n}^-,p_n)\in\M^{M}.
    \end{equation}
    We investigate the length of M. First, \(b_i^+\) is given by \(p_{i-1}\)
    and \(b_i^-\) via~\eqref{eq:c1cond}.
    Second, for the segments \(i = 2,\ldots,n\) we can further omit \(p_{i-1}\)
    since the value also occurs in the segment~$i$ as last entry.
    Finally, the first segment contains
    the additional value of~\(b_0^+\) that is not fixed by \(C^1\) constraints.
    The first segment is thus composed of $K+1$ control points, while the $n-1$ remaining
    segments are determined by \(K-1\) points. In total we obtain~\(M=n(K-1)+2\)
    control points to optimize.
    
    Minimizing $A(\vect{b})$ alone leads to the trivial solution, for any set
    of control points $\vect{b}=(x,\ldots,x)$, $x\in\M$,
    and this is why the fitting term from~\eqref{eq:EB} is important.

    \paragraph{Fitting ($0 < \lambda < \infty$).}
    If the segment start and end points are obstructed by noise or allowed to
    move, we employ a fitting scheme to balance the importance given to the data points $d_0,\ldots,d_n$.
    Equation~\eqref{eq:EB} reads
    \begin{equation}\label{eq:ApproxFct}
     \argmin_{\vect{b}\in\Gamma_{\bspline}} \tilde A(\vect{b}), \quad
     \tilde A(\vect{b}) \coloneqq
        A(\vect{b}) + \frac{\lambda}{2}\sum_{i=0}^n d^2_{\M}(d_i,p_i),
    \end{equation}
    where \(\lambda \in \R^+\) sets the priority 
    to either the data term (large \(\lambda\)) or the mean squared acceleration
    (small \(\lambda\)) within the minimization.
    The gradient of the data term is given in~\cite{Karcher1977},
    and the gradient of \(\tilde A\) is given by
    \[
      \nabla_{\M^M,x}\tilde A(\vect{b})
      =
      \begin{cases}
        \nabla_{\M^M,x}A(\vect{b}) -\lambda\log_{p_i}d_i
           & \text{ if } x=p_i,\ i=0,\ldots,n,\\
           \nabla_{\M^M,x}A(\vect{b})&\text{ otherwise.}
      \end{cases}
    \]
    
    \paragraph{Interpolation ($\lambda \to \infty$).}
    For interpolation we assume that the start point $p_{i-1}$ and end point~$p_{i}$ 
    of the segments $i=1,\dots,n$ are fixed to given data $d_{i-1}$
    and~$d_{i} \in\mathcal M$, respectively.
    The optimization of the discrete mean squared acceleration $A(\vect{b})$ reads
    \begin{equation}\label{eq:IPFct}
        \argmin_{\vect{b}\in\Gamma_{\bspline}} A(\vect{b}) \quad \text{s.~t. } p_i=d_i,\ i=0,\ldots,n.
    \end{equation}
    Since the $p_i$ are fixed by constraint, they can be omitted from the
    vector~$\vect{b}$. We obtain
    \[
      \vect{b} =
      (b_0^+,b_2^0,\ldots,b_{1}^-,b_2^1,\ldots,b_{n}^-) \in\M^{M'}
    \]
    Since there are $n+1$ additional constraints, the minimization is hence performed on the product 
    manifold $\M^{M'}$, $M'=M-(n+1)=n(K-2)+1$.
  %
  % Adjoint Jacobi
  \subsection{Adjoint Jacobi fields}\label{subsec:AdjJacobi}
    In the Euclidean space $\mathbb R^m$, the adjoint operator $T^*$ of a
    linear bounded operator $T\colon \mathbb R^m\to\mathbb R^q$ is the operator
    fulfilling
    \begin{equation*}
      \langle T(x),y \rangle_{\mathbb R^q} = \langle x,T^*(y) \rangle_{\mathbb R^m}
      ,\quad\text{ for all }x\in\mathbb R^m,\ y\in\mathbb R^q.
    \end{equation*}
  
    The same can be defined for a linear operator $S\colon T_x\M\to T_y\M$,
    $x,y\in\mathcal M$, on a $m$-dimensional Riemannian manifold~\(\M\). The adjoint
    operator~\(S^*\colon T_y\M\to T_x\M\) satisfies
    \begin{equation*}
      \langle S(\eta),\nu \rangle_y = \langle \eta, S^*(\nu)\rangle_x,
      \quad\text{ for all } \eta\in T_x\M,\ \nu\in T_y\M.
    \end{equation*}
    
    We are interested in the case where $S$ is the differential
    operator $D_x$ of a geodesic~$F(x) = g(t;x,y)$ for some fixed $t\in\mathbb R$ and $y\in\M$.
    The differential \(D_xF\colon T_x\M\to T_{F(x)}\M\) can be written as 
    \[
      D_xF[\eta]
      =
      J_{g,\eta}(t)
      =
      \sum_{\ell=1}^{m} \langle \eta,\xi_\ell\rangle_x\alpha_\ell\Xi_\ell(t),
    \]
    where \(\alpha_\ell\) are the coefficients of the Jacobi
    field~\eqref{eq:DxGeoJacobi}, and \(\xi_\ell, \Xi_\ell(t)\) are given as in
    Lemma~\ref{lem:symspace}.
    To derive the \defTerm{adjoint differential} \((D_xF)^*\colon T_{F(x)}\M\to T_x\M\)
    we observe that for any $\nu\in T_{F(x)}\M$ we have
    \begin{equation*}
      \langle D_xF[\eta],\nu\rangle_{F(x)}
      = \sum_{\ell=1}^m \langle \eta,\xi_\ell\rangle_x\alpha_\ell \langle \Xi_\ell(t),\nu\rangle_{F(x)}
      = \Bigl\langle \eta, \sum_{\ell=1}^m \langle\Xi_\ell(t),\nu\rangle_{F(x)}\alpha_\ell\xi_\ell \Bigr\rangle_x.
    \end{equation*}
    Hence the adjoint differential is given by
    \[
      (D_xF)^*[\nu] = \sum_{\ell=1}^m
        \langle \nu,\Xi_\ell(t)\rangle_{F(x)}\alpha_\ell\xi_\ell,
      \qquad \nu\in T_{F(x)}\M.
    \]
    We introduce
    the \emph{adjoint Jacobi field}~\(J^*_{F,\nu}\colon\mathbb R \to
    T_x\mathcal M\), \(\nu\in T_{F(x)}\M\) as
    \[
      J^*_{F,\nu}(t) = \sum_{\ell=1}^m \langle \nu,\Xi_\ell(t)\rangle_{F(x)} J^*_{F,\Xi_\ell(t)}(t)
      = \sum_{\ell=1}^m \langle \nu,\Xi_\ell(t)\rangle_{F(x)}\alpha_\ell\xi_\ell.
    \]
    Note that evaluating the adjoint Jacobi field~$J^*$ involves
    the same transported\\
    frame~\(\{\Xi_1(t),\ldots\Xi_m(t)\}\)
    and the same coefficients $\alpha_\ell$ as the Jacobi field~$J$,
    which means that the evaluation of the adjoint is in no way inferior to
    the Jacobi field itself.

    The adjoint $D^*$ of the differential is useful in particular, when computing the
    gradient~\(\nabla_{\M}(h\circ F)\) of the composition of~\(F\colon\mathcal M \to\mathcal M\)
     with~\(h\colon\M\to\R\).
    Setting $y \coloneqq F(x) \in \M$, we obtain for any \(\eta\in T_x\mathcal M\) that
    \begin{align*}
      \langle \nabla_{\M,x}(h\circ F)(x),\eta\rangle_x
      &= D_x(h\circ F)[\eta] \\
      &= D_{F(x)}h\bigl[ D_x F[\eta] \bigr] \\
      &= \langle \nabla_{\M,y}h(y), D_xF[\eta] \rangle_{y} \\
      &= \Bigl\langle (D_xF)^*\bigl[\nabla_{\M,y} h(y)\bigr], \eta \Bigr\rangle_x.
    \end{align*}
    
    Especially for the evaluation of the gradient of the composite function
    $h\circ F$ we obtain
    \[
      \nabla_{\M,x}(h\circ F)(x) =
      (D_xF)^*\bigl[\nabla_{\M,y} h(y)\bigr] = J^*_{F,\nabla_{\M,y}h(y)}(t).
    \]
    The main advantage of this technique appears in the case of composite
    functions, \ie, of the form \(h\circ F_1\circ F_2\) (the generalization to composition with $K$ functions is straightforward). 
    The gradient
    \(\nabla_{\M,x} (h\circ F_1 \circ F_2)(x) \) now reads,
    \[
      \nabla_{\M,x} (h\circ F_1 \circ F_2)(x) = (D_xF_2)^*\bigl[\nabla_{\M,y_2} h\circ F_1(y_2)\bigr] = J^*_{F_2,\nabla_{\M,y_2} h \circ F_1(y_2))}(t).
    \]
    The recursive computation of $\eta^\bc 3 =\nabla_{\M,x} h (x)$ is then given by the following algorithm
  \begin{eqnarray*}
    \eta^\bc 1 &=& \nabla_{\M,y_1}h(y_1), \\ 
    \eta^\bc 2 &=& J^*_{F_1,\eta^\bc 1}(t), \\ 
    \eta^\bc 3 &=& J^*_{F_2,\eta^\bc 2}(t).
  \end{eqnarray*}
  \begin{exmp}
    For \(h=d_2^2\colon \M^3\to \mathbb R\) we know \(\nabla_{\M^3}h\) by
    Lemma~3.1 and Lemma~3.2 from~\cite{Bacak2016}. Let~\(t_1,t_2,t_3 \in [0,1]\)
    be time points, and $\vect{b}\in\M^M$ be a given (sub)set of the control points
      of a (composite) Bézier curve \(\bspline\).
      We define \(F\colon\M^M\to\M^3$, $\vect{b} \mapsto F(\vect{b}) =
      (\bspline(t_1),\bspline(t_2),\bspline(t_3))\) as the evaluations
      of~$\bspline$ at the three given time points.
      The composition $h\circ F$ hence consists of (in order of evaluation) the
      geodesic evaluations of the De Casteljau algorithm, the mid point function
      for the first and third time point and a distance function.
      
      The recursive evaluation of the gradient starts with the
      gradient of the distance function.
      Then, for the first and third arguments, a mid point Jacobi field
      is applied. The result is plugged into the last geodesic evaluated within the
      De Casteljau “tree” of geodesics. At each geodesic, a tangent vector at
      a point $g(t_j;a,b)$ is the input for two adjoint Jacobi fields, one
      mapping to $T_a\M$, the other to $T_b\M$. This information is available
      throughout the recursion steps anyways. After traversing this tree backwards,
      one obtains the required gradient of $h\circ F$.
    \end{exmp}
    
    Note also that even the differentiability constraint \eqref{eq:c1cond}
    yields only two further (most outer) adjoint Jacobi fields,
    namely~\(J^*_{g(2,b_i^-,p_i),\nabla_{\M,b_i^+}\bspline(t_j)}(2)\)
    and~\(J^*_{\tilde g(2,b_i^-,p_i),\nabla_{\M,b_i^+}\bspline(t_j)}(2)\).
    They correspond to variation of the start point \(b_i^-\) and the end
    point~\(p_i\), respectively as stated in~\eqref{lem:derC1}.
  %
  % GradDesc.
  \subsection{A gradient descent algorithm}\label{subsec:GradDesc}
    To address~\eqref{eq:ApproxFct} or~\eqref{eq:IPFct},
    we use a gradient descent algorithm, as described in~\cite[Ch.~4]{Absil2008}. For completeness
    the algorithm is given in Algorithm~\ref{alg:GradDesc}.
    \begin{algorithm}[t]
      \caption[]{Gradient descent algorithm on a manifold~\(\mathcal N=\M^M\)}
      \label{alg:GradDesc}
      \begin{algorithmic}
        \STATE \textbf{Input. }%
          $F\colon\mathcal N\to\mathbb R$,
          its gradient \(\nabla_{\mathcal N} F\), $x^{(0)}\in\mathcal N$,
          step sizes $s_k>0, k\in\mathbb N$.
        \STATE \textbf{Output:} $\hat x \in \mathcal N$
        \STATE $k \gets 0$
        \REPEAT
        \STATE Perform a gradient descent step
        \(
          x^{(k+1)} \coloneqq \exp_{x^{(k)}}
          \bigl(
            -s_k\nabla_{\mathcal N}F(x^{(k)})
          \bigr)
        \)
        \STATE $k\gets k+1$
        \UNTIL a stopping criterion is reached
        \RETURN $\hat{x} \coloneqq x^{(k)}$
       \end{algorithmic}
     \end{algorithm}
     The step sizes are given by the Armijo line search condition 
     presented in~\cite[Def.~4.2.2]{Absil2008}. Let $\mathcal N$ be a Riemannian manifold, \(x=x^{(k)}\in\mathcal N\) be an iterate of the
     gradient descent, and \(\beta,\sigma\in(0,1), \alpha>0\). Let \(m\) be the
     smallest positive integer such that
     \begin{equation}\label{eq:armijo}
       F(x) - F\bigl(
         \exp_{x}(-\beta^m\alpha\nabla_{\mathcal N}F(x))
       \bigr)
       \geq
       \sigma\beta^m\alpha\lVert \nabla_{\mathcal N}F(x)\rVert_x.
     \end{equation}
     We set the step size to \(s_k \coloneqq \beta^m\alpha\)
     in Algorithm~\ref{alg:GradDesc}.

     As a stopping criterion we use a maximal number \(k_{\max}\) of iterations or
     a minimal change per iteration \(d_{\mathcal
     N}(x^{k},x^{k+1})<\varepsilon\). In practice, this last criterion is matched first.

     The gradient descent algorithm converges to a critical
     point if the function $F$ is convex~\cite[Sec.~4.3 and 4.4]{Absil2008}.
     The mid-points
     model~\eqref{eq:secondOrderDifference} posesses two advantages: (i) the
     complete discretized MSA~\eqref{eq:discr_obj} consists of (chained)
     evaluations of geodesics and a distance function, and (ii) it reduces
     to the classical second order differences on the Euclidean space. However, this model
     is not convex on general manifolds.
     An example is given in the arXiv preprint (version 3)
     of~\cite{Bacak2016}\footnote{see \href{https://arxiv.org/abs/1506.02409v3}%
     {arxiv.org/abs/1506.02409v3}}, Remark 4.6. Another possibility (also
     reducing to classical second order differences in Euclidean space) is the
     so-called Log model (see, \eg,~\cite{Boumal2013})
     \begin{align*}
       d_{2,\text{Log}}[x,y,z]  = \lVert \log_yx+\log_yz\rVert_y.
     \end{align*}
     The (merely technical) disadvantage of the Log model is
     that the computation
     of the gradient involves further Jacobi fields than the one presented above,
     namely to compute the differentials of the logarithmic map both with respect
     to its argument~$D_x\log_yx$ as well as its base point~$D_y\log_yx$.
     Still, these can be given in closed form for symmetric Riemannian
     manifolds~\cite[Th.~7.2]{Bermann2017a}\cite[Lem.2.3]{Persch2018}.
     To the best of our knowledge, the joint convexity of the Log model
     in $x$, $y$ and $z$ is still an open question.
     
% Examples
% ---
\section{Examples}\label{sec:Examples}
  In this section, we provide several examples of our algorithm
  applied to the fitting problem~\eqref{eq:EB}.
  
  We validate it first on the Euclidean space and verify that it retrieves
  the natural cubic smoothing spline. We then present examples on the
  sphere $\mathbb{S}^2$ and the special orthogonal group $\mathrm{SO}(3)$.
  We compare our results with the fast algorithm of Arnould~\etal~\cite{Arnould2015},
  generalized to fitting, and the so-called blended
  cubic splines from~\cite{Gousenbourger2018}.
  The control points of the former are obtained by generalizing the optimal
  Euclidean conditions of~\eqref{eq:EB} (in $\R^m$, this is a 
  linear system of equations) to the manifold setting; 
  the curve is afterwards reconstructed by a classical De Casteljau algorithm.
  In the latter, the curve is obtained as a blending of solutions computed
  on carefully chosen tangent spaces, \ie, Euclidean spaces.
  
  The following examples were implemented in MVIRT~\cite{Bergmann2017}\footnote{
  open source, available at \url{http://ronnybergmann.net/mvirt/}}, and the
  comparison implementations from~\cite{Arnould2015} and~\cite{Gousenbourger2018} use
  Manopt~\cite{Manopt}\footnote{ open source, available
  at~\url{http://www.manopt.org}}. Note that both toolboxes use a very similar
  matrix structure and are implemented in Matlab, such that the results can
  directly be compared.

  \subsection{Validation on the Euclidean space}\label{subsec:eucl}
  As a first example, we perform a minimization on the Euclidean space~\(\mathcal
  M=\mathbb R^3\).
  A classical result on $\R^m$ is that the curve $\gamma$ minimizing~\eqref{eq:E} is the natural
  ($\C^2$) cubic spline when $\Gamma$ is a Sobolev space $H^2(t_0,t_n)$.
  We compare our approach to the tangential linear system approach derived 
  in~\cite{Arnould2015} in order to validate our model. 
  We use the following data points
  \begin{align}\label{eq:unitPi}
    d_0 = \begin{bmatrix}
      0\\0\\1
    \end{bmatrix},\quad
    d_1 = \begin{bmatrix}
      0\\-1\\0
    \end{bmatrix},\quad
    d_2 = \begin{bmatrix}
      -1\\0\\0
    \end{bmatrix},\quad
    d_3 = \frac{1}{\sqrt{82}}\begin{bmatrix}
      0\\-1\\-9
  \end{bmatrix},
  \end{align}
  as well as the control points $p_i = d_i$, and
  \begin{equation}\label{eq:unitbi}
  \begin{aligned}
    b_0^+ &= \exp_{p_0} \frac{\pi}{8\sqrt{2}}
    \begin{bmatrix}1\\-1\\0\end{bmatrix},
    &b_1^+ &= \exp_{p_1} -\frac{\pi}{2\sqrt{2}}
    \begin{bmatrix}1\\0\\1\end{bmatrix},
    &b_1^- &= g(2;b_1^+,p_1),\\
    b_2^+ &= \exp_{p_2}\frac{\pi}{2\sqrt{2}}
    \begin{bmatrix}
    0\\1\\-1
    \end{bmatrix},
    &b_3^- &= \exp_{p_3} \frac{\pi}{8}\begin{bmatrix}-1\\0\\0\end{bmatrix},
    &b_2^- &= g(2;b_2^+,p_2),
  \end{aligned}
  \end{equation}
  where the exponential map and geodesic on \(\mathbb R^3\)
  are actually the addition
  and line segments, respectively. Note that, by construction, the 
  initial curve is continuously differentiable
  but (obviously) does not minimize~\eqref{eq:E}. The parameter \(\lambda\) is set to \(50\).
  The MSA of this initial curve \(\tilde A(\vect{b})\) is
  approximately \(18.8828\).

  The data points are given to the algorithm of~\cite{Arnould2015} to reconstruct
  the optimal Bézier curve $\bspline(t)$, which  is the natural $\C^2$ cubic spline.
  The result is shown in
  Fig.~\ref{subfig:EuclideanTest:Curves} together with the first order
  differences along the curve in Fig.~\ref{subfig:EuclideanTest:firstOrder}
  as a dotted curve.

  \begin{figure}[tbp]\centering
  \begin{subfigure}[b]{.48\linewidth}\centering
    ~\hspace{-.8cm}\includegraphics{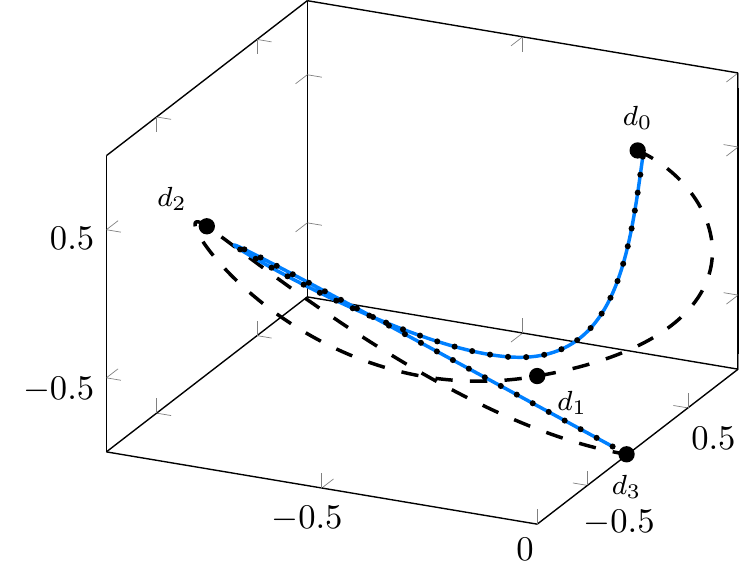}
    \caption{Bézier curves.}
    \label{subfig:EuclideanTest:Curves}
  \end{subfigure}
  \begin{subfigure}[b]{.48\textwidth}\centering
    ~\hspace{-1.3cm}\includegraphics{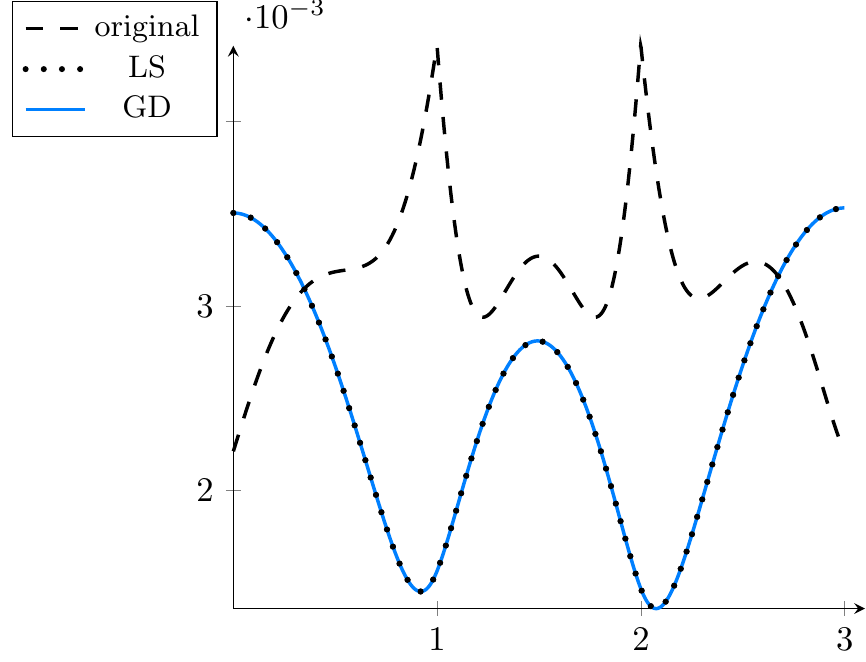}
    \caption{First order differences.}
    \label{subfig:EuclideanTest:firstOrder}
  \end{subfigure}
  \caption{%
    The initial interpolating Bézier curve in \(\mathbb R^3\) (dashed,
    \subref{subfig:EuclideanTest:Curves}) with an MSA of~\(18.8828\) is 
    optimized by the linear system method (LS) from~\cite{Arnould2015}
    (dotted) and by the proposed gradient descent (GD, solid blue).
    As expected, both curve coincide, with
    an MSA of~\(4.981218\). The correspondance is further illustrated with
    their first order differences
    in~(\subref{subfig:EuclideanTest:firstOrder}).
  }
  \end{figure} 

  The set of control points \((p_0,b_0^+,\dots,b_3^-,p_3)\) is
  optimized with our proposed method. We discretize the second
  order difference using \(N=1600\) points.
  The resulting curve and the first order difference plots
  are also obtained using these sampling values and a first order forward
  difference.
  The gradient descent algorithm from Algorithm~\ref{alg:GradDesc}
  employs the Armijo rule~\eqref{eq:armijo} setting~\(\beta=\frac{1}{2}\),~\(\sigma=10^{-4}\),
  and~\(\alpha=1\).
  The stopping criteria are \(\sum_{i=1}^{1600} d_{\mathcal M}(x^{(k)}_i,x^{(k)}_i) < \epsilon = 10^{-15}\) or
  \(\lVert \nabla_{\mathcal M^n}\tilde A\Vert_2 < 10^{-9}\), and the algorithm 
  stops when one of the two is met. For this example, the
  first criterion stopped the algorithm, while the norm of the gradient was of
  magnitude $10^{-6}$.

  Both methods improve the initial functional value of \(\tilde A(\vect{b}) \approx
  18.8828\) to a value of~\(\tilde A(\vect{b}_{\text{min}}) \approx 4.981218\).
  Both the linear system approach and the gradient descent perform equally.
  The difference of objective value is
  \(2.4524\times10^{-11}\) smaller for the gradient descent, and the maximal
  distance of any sampling point of the resulting curves is of size
  \(4.3\times10^{-7}\). Hence, in the Euclidean space, the proposed gradient descent
  yields the natural cubic spline, as  one would expect.

  \subsection{Examples on the sphere~$\mathbb{S}^2$}
  \paragraph{Validation on a geodesic.}
  As a second example with a known minimizer,
  we consider the manifold~\(\mathcal
  M=\mathbb S^2\), \ie, the two-dimensional unit sphere embedded
  in~\(\mathbb R^3\), where geodesics are great arcs. We use the data points
  \[
    d_0 = \begin{bmatrix}
      0 \\ 0 \\ 1
    \end{bmatrix},\quad
    d_1 = \begin{bmatrix}
      0 \\ 1 \\ 0
    \end{bmatrix},\quad
    d_2 = \begin{bmatrix}
      0 \\ 0 \\ -1
    \end{bmatrix}
  \]
  aligned on the geodesic connecting the
  north pole \(p_0\) and the south pole~\(p_2\), and running through a point~\(p_1\)
  on the equator. We define the control points of the cubic Bézier curve as follows:
  \begin{align*}
    x_0 = \frac{1}{\sqrt{6}}\begin{bmatrix}1 \\ 1 \\ 2\end{bmatrix},
    && 
    b_0^+ = \Exp{p_0}[3\Log{p_0}[x_0]] ,
    &&
    b_1^- = \frac{1}{\sqrt{6}}\begin{bmatrix}1 \\ 2 \\ 1\end{bmatrix},
    \\
    x_2 = \frac{1}{\sqrt{6}}\begin{bmatrix}-1 \\ 1 \\ -2\end{bmatrix}, 
    && 
    b_1^+ = \frac{1}{\sqrt{6}}\begin{bmatrix}-1 \\ 2 \\ -1\end{bmatrix},
    &&
    b_2^- = \Exp{p_2}[\frac{1}{3}\Log{p_0}[x_2]],
  \end{align*}
  where $x_0$ and $x_2$ are temporary points and $p_i = d_i$. We
  obtain two segments smoothly connected since \(\log_{p_1}b_1^- = -\log_{p_1}b_1^+\).
  The original curve is shown in Fig.~\ref{subfig:S2Seg:Orig}, where especially
  the tangent vectors illustrate the different speed at $p_i$.
  
  The control points are optimized
  with our interpolating model, \ie, we fix the start and end
  points~\(p_0,p_1,p_2\) and minimize~\(A(\vect{b})\).
  
  The curve, as well as the second and first order differences, is sampled
  with~\(N=201\) equispaced points.  
  The parameters of the Armijo rule are again set
  to~\(\beta=\frac{1}{2}\),~\(\sigma=10^{-4}\), and~\(\alpha=1\).
  The stopping criteria are slightly relaxed to \(10^{-7}\) for the distance
  and \(10^{-5}\) for the gradient, because of the sines and cosines involved
  in the exponential map.
  \begin{figure}\centering
    \begin{subfigure}[b]{.31\textwidth}\centering
      \includegraphics[width=.9\textwidth]{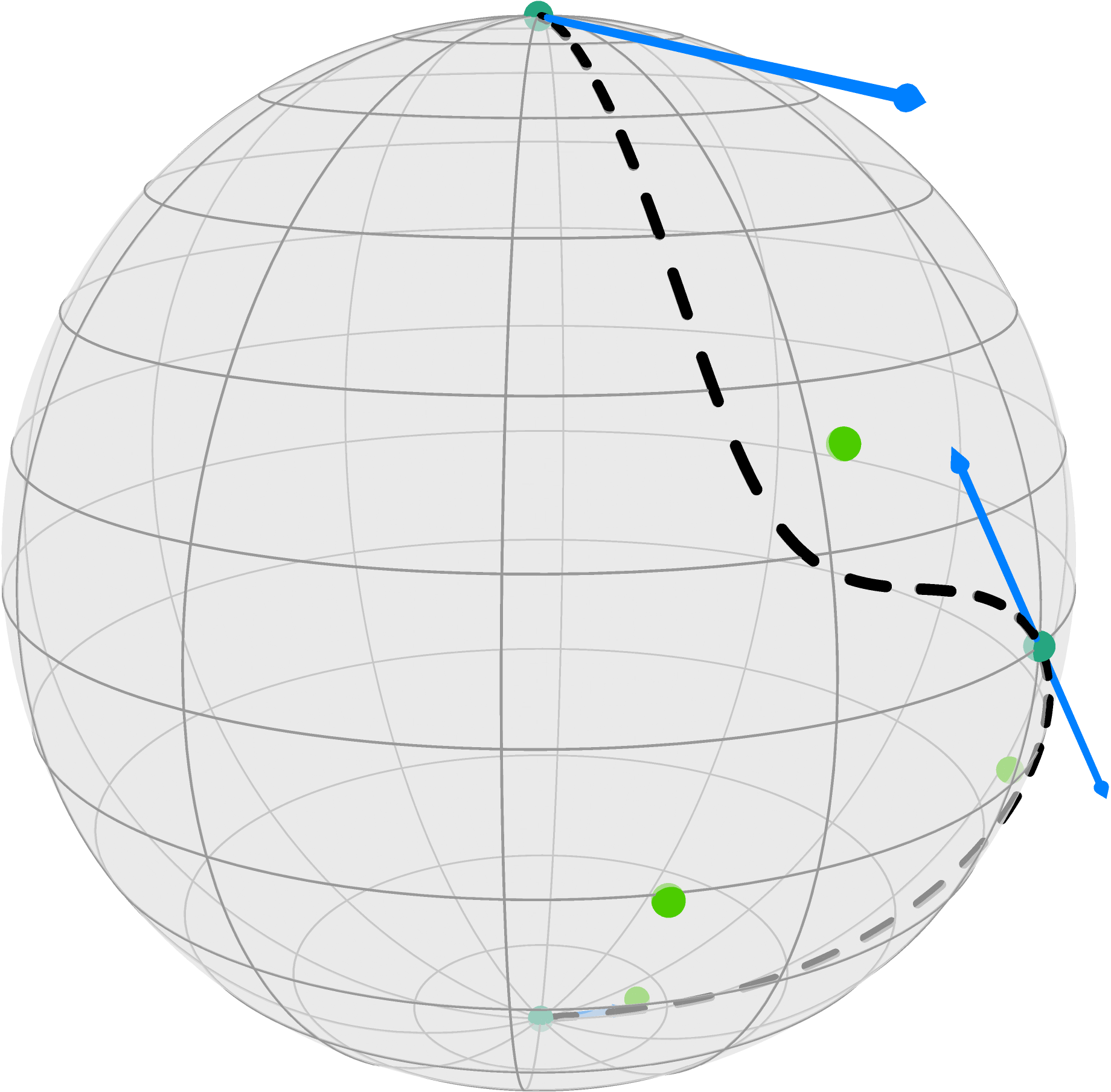}
      \caption{The original curve.}
      \label{subfig:S2Seg:Orig}
    \end{subfigure}
    \begin{subfigure}[b]{.31\textwidth}\centering
      \includegraphics[width=.9\textwidth]{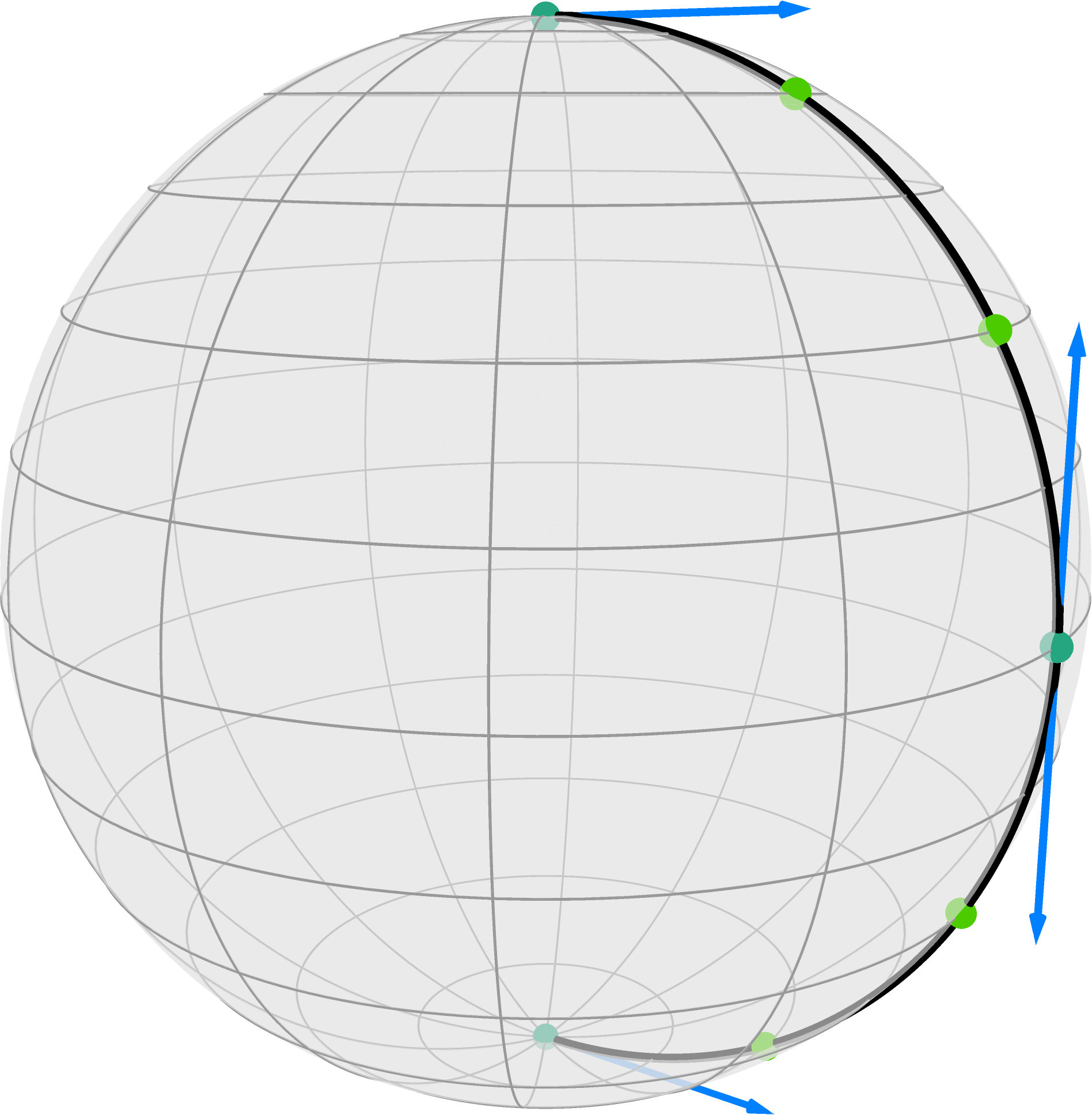}
      \caption{The minimized curve.}
      \label{subfig:S2Seg:Result}
    \end{subfigure}
    \begin{subfigure}[b]{.36\textwidth}\centering
      \includegraphics{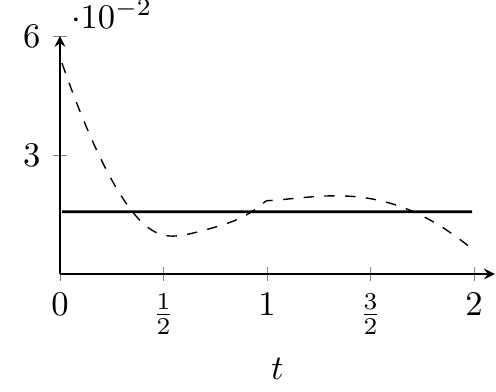}
      \caption{Absolute first order difference.}
      \label{fig:S22SegPlot}
    \end{subfigure}
    \caption{The initial curve (dashed, \subref{subfig:S2Seg:Orig}) results in a
        geodesic~(solid, \subref{subfig:S2Seg:Result}) when minimizing the
        discrete acceleration. This can also be seen on the first order
        differences~(\subref{fig:S22SegPlot}).}
    \label{fig:S22Seg}
  \end{figure}

  The result is shown in Fig.~\ref{subfig:S2Seg:Result}. Since the minimizer is
   the geodesic running from \(p_0\) to \(p_2\)
  through \(p_1\), we measure the perfomance first by looking at the
  resulting first order difference, which is constant, as can be seen in Fig.~\ref{fig:S22SegPlot}.
  As a second validation, we observe that the maximal distance of the resulting
  curve to the geodesic is of~\(2.2357\times10^{-6}\). These evaluations again
  validate the quality of the gradient descent.

  \paragraph{Effect of the data term.}
  As a third example we investigate the effect of \(\lambda\) in the
  fitting model. We consider the data points
  \[
    d_0 = \begin{bmatrix}
      0&0&1
    \end{bmatrix}^\tT, \qquad
    d_1 = \begin{bmatrix}
      0&-1&0
    \end{bmatrix}^\tT, \qquad
    d_2 = \begin{bmatrix}
      -1&0&0
    \end{bmatrix}^\tT, \qquad
    d_3 = \begin{bmatrix}
      0&0&-1
    \end{bmatrix}^\tT,
  \]
  as well as the control points $p_i = d_i$, and
  \begin{align*}
    b_0^+ &= \Exp{p_0}[\frac{\pi}{8\sqrt{2}}
        \begin{bmatrix} 1  \\ -1  \\ 0 \end{bmatrix}], &
    b_1^+ &= \Exp{p_1}[-\frac{\pi}{4\sqrt{2}}
        \begin{bmatrix} -1  \\ 0  \\ 1 \end{bmatrix}], \\
    b_2^+ &= \Exp{p_2}[\frac{\pi}{4\sqrt{2}}
        \begin{bmatrix} 0  \\ 1  \\ -1 \end{bmatrix}], &
    b_3^- &= \Exp{p_3}[-\frac{\pi}{8\sqrt{2}}
        \begin{bmatrix} -1  \\ 1  \\ 0 \end{bmatrix}].
  \end{align*}
  The remaining control points $b_1^-$ and $b_2^-$ are given by
  the~$\C^1$ conditions~\eqref{eq:c1cond}.
  The corresponding cuve \(\mathbf B(t)\), \(t\in [0,3]\), is shown in
  Fig.~\ref{fig:S2-3Seg} in dashed black. When computing a minimal MSA curve that
  interpolates~\(\mathbf B(t)\) at the points \(p_i\), \(i=0,\ldots,3\), the
  acceleration is not that much reduced, see the blue curve in Fig.~\ref{subfig:S2-3Seg:IP}.

  For fitting, we consider different values of \(\lambda\)
  and the same parameters as for the last example. 
  The optimized curve fits the data points closer and closer
  as $\lambda$ grows, and the limit
  \(\lambda\to\infty\) yields the interpolation case. On the other hand smaller
  values of $\lambda$ yield less fitting, but also a smaller value of the
  mean squared acceleration. In the limit case, \ie, $\lambda=0$, the
  curve (more precisely the control points of the Bézier curve) just follows the
  gradient flow to a geodesic.
    
  The results are collected in Fig.~\ref{fig:S2-3Seg}.
  In Fig.~\ref{subfig:S2-3Seg:IP}, the original curve (dashed) is shown together
  with the solution of the interpolating model (solid blue) and the gradient
  flow result, \ie, the solution of the fitting  model (solid black)
  with~\(\lambda=0\). Fig.~\ref{subfig:S2-3Seg:Approx} illustrates the effect
  of~\(\lambda\) even further with a continuous variation of \(\lambda\)
  from a large value of~$\lambda=10$ to a small value of
  $\lambda=0.01$. The image is generated by sampling this range with $1000$
  equidistant values colored in the colormap \texttt{viridis}.
  Furthermore, the control points are also shown in the same color. 
  
  Several corresponding functional values, see Tab.~\ref{tab:S2IPApprValues}, further
  illustrate that with smaller values of \(\lambda\) the discretized MSA also
  reduces more and more  to a geodesic. For \(\lambda=0\) there is no coupling
  to the data points and hence the algorithm does not restrict the position of
  the geodesic. In other words, any choice for the control points that yields a
  geodesic, is a solution. Note that the gradient flow still chooses a reasonably near
  geodesic to the initial data (Fig.~\ref{subfig:S2-3Seg:IP}, solid black).
  
  \begin{figure}[tbp]
    \begin{subfigure}[t]{.445\textwidth}
      \includegraphics[width=6.5cm]{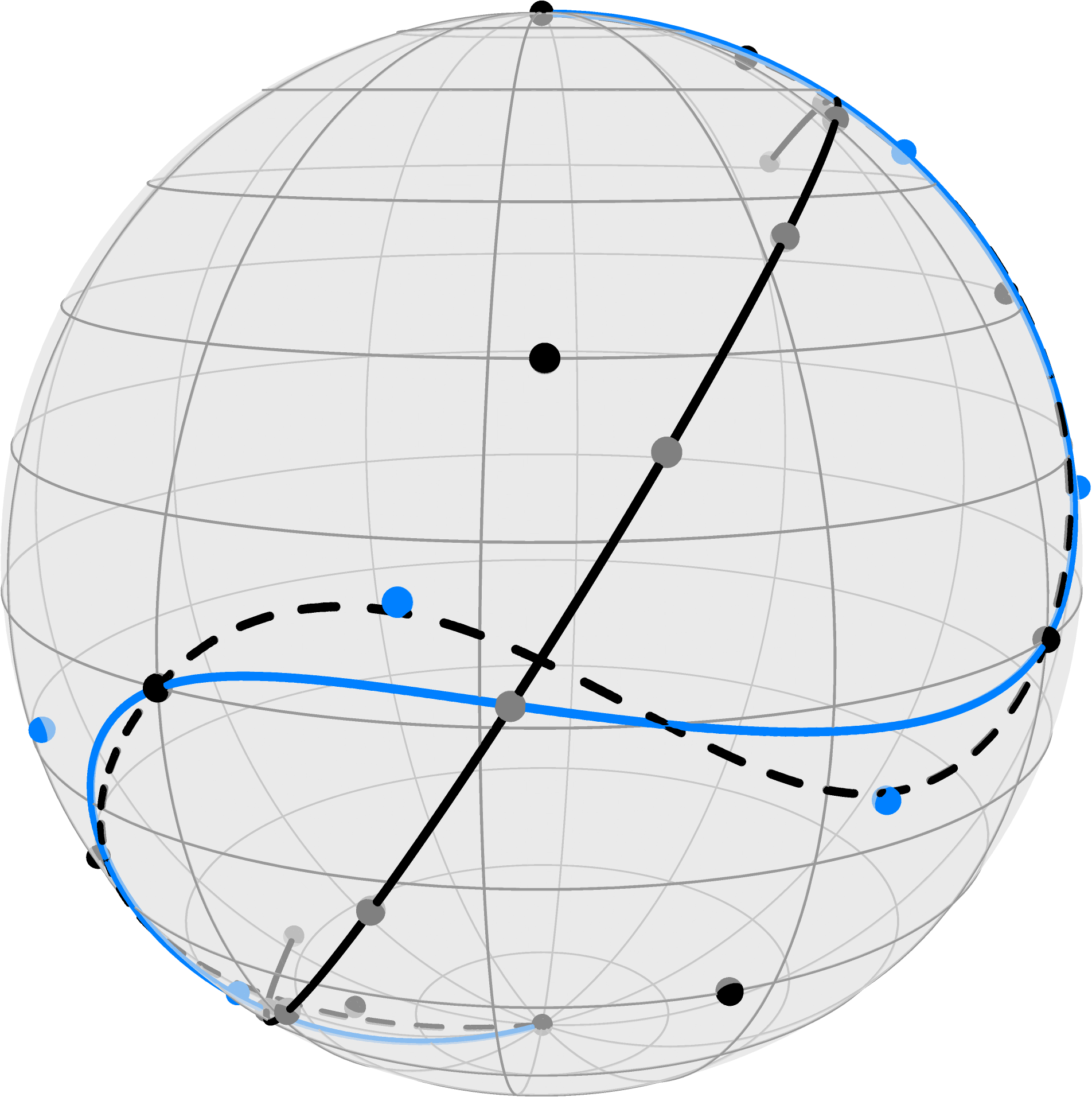}
      \caption{Initial (dashed), interpolation (solid blue) and fitting with \(\lambda=0\) (solid black).}
      \label{subfig:S2-3Seg:IP}
    \end{subfigure}
    \begin{subfigure}[t]{.535\textwidth}
      \includegraphics{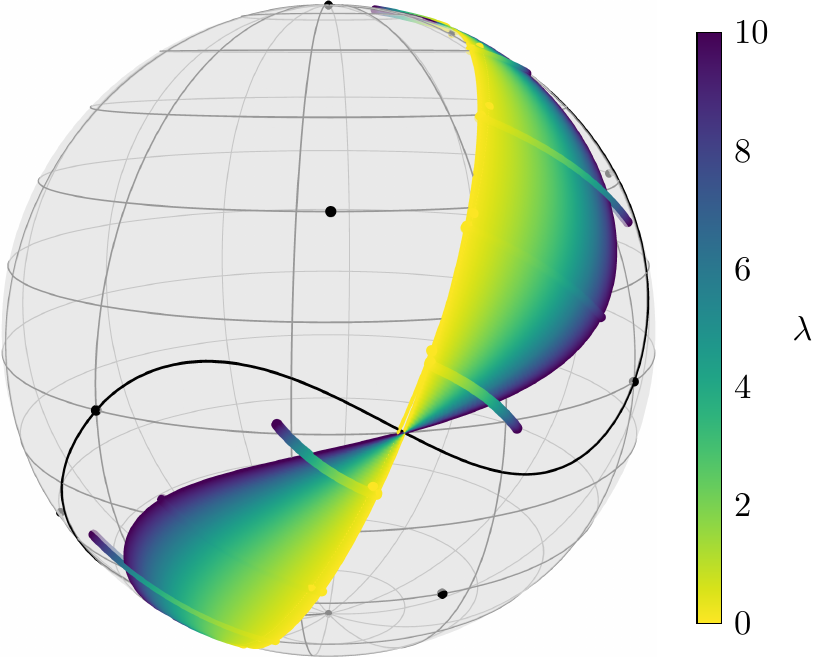}
      \caption{Reducing the data term from $\lambda=10$ (violet)
      down to $\lambda=0.01$ (yellow) in $1000$ equidistant steps.}
      \label{subfig:S2-3Seg:Approx}
    \end{subfigure}
    \caption{The composite B\'ezier curves are composed of three segments.
      The initial curve (dashed, \subref{subfig:S2-3Seg:IP}) is optimized with
      the interpolating model (solid blue, \subref{subfig:S2-3Seg:IP}) as well as with the fitting
      model for a continuum of values of
      $\lambda$, from $\lambda=10$ (violet, \subref{subfig:S2-3Seg:Approx})
      to $\lambda=0.01$ (yellow, \subref{subfig:S2-3Seg:Approx}).
      The limit case where $\lambda = 0$ yields an unconstrained geodesic (solid black, \subref{subfig:S2-3Seg:Approx}).
      }
    \label{fig:S2-3Seg}
  \end{figure}
  \begin{table}[tbp]\centering
    \setlength{\tabcolsep}{12pt}
    \begin{tabular}{S[table-format=2.1] S[table-format=2.4e+2]}
      \toprule
      {$\lambda$}  & {$\tilde A(\vect{b})$} \\
      \midrule
      {orig}     & 10.6122 \\
      {$\infty$}  & 4.1339 \\
      10      & 1.6592 \\
      1      & 0.0733 \\
      0.1     & 0.0010 \\
      0.01     & 1.0814e-5 \\
      0.001     & 1.6240e-7 \\
      0       & 3.5988e-9 \\
      \bottomrule
    \end{tabular}
    \caption{Functional values for the three-segment composite
      Bézier curve on the sphere and different values of $\lambda$.
      Note that the case \(\lambda=\infty\) corresponds to interpolation.}
    \label{tab:S2IPApprValues}
  \end{table}

  \paragraph{Comparison with the tangential solution.}
  In the Euclidean space,
  the method introduced in~\cite{Arnould2015,Gousenbourger2018} and the
  gradient descent proposed here yield the same curve, \ie, the natural cubic spline.
  On Riemannian manifolds, however, all approaches approximate the optimal solution. Indeed,
  their method provides a solution by working in different tangent spaces,
  and ours minimizes a discretization of the objective functional.

  In this example we take the same data points \(d_i\) as in \eqref{eq:unitPi} now
  interpreted as points on \(\mathcal M=\mathbb S^2\). 
  Note that the control points constructed in~\eqref{eq:unitbi}
  still fit to the sphere since each \(b_i^{\pm}\) is built with a vector
  in~\(T_{p_i}\mathcal M\) and using the corresponding exponential map.
  The result is shown in Fig.~\ref{subfig:S2Compare:Curves}.
  The norm of the first order differences is given in Fig.~\ref{subfig:S2Compare:firstOrder}. 
  The initial  curve (dashed black) has an objective value of~\(10.9103\).
  The tangent version (dotted) reduces this
  value to \(7.3293\). The proposed method (solid blue)
  yields a value of \(2.7908\), regardless of whether the starting curve is
  the initial one, or the one already computed by~\cite{Arnould2015}.
  Note that, when \(p_3 = [0,0,-1]^{\tT}\), the tangent version is even not able to
  compute any result, since the construction is performed in the tangent space
  of \(p_0\) and since~\(\log_{p_0}p_3\) is not defined.

  \begin{figure}\centering
    \begin{subfigure}[b]{.445\textwidth}\centering
      \includegraphics[width=.95\textwidth]{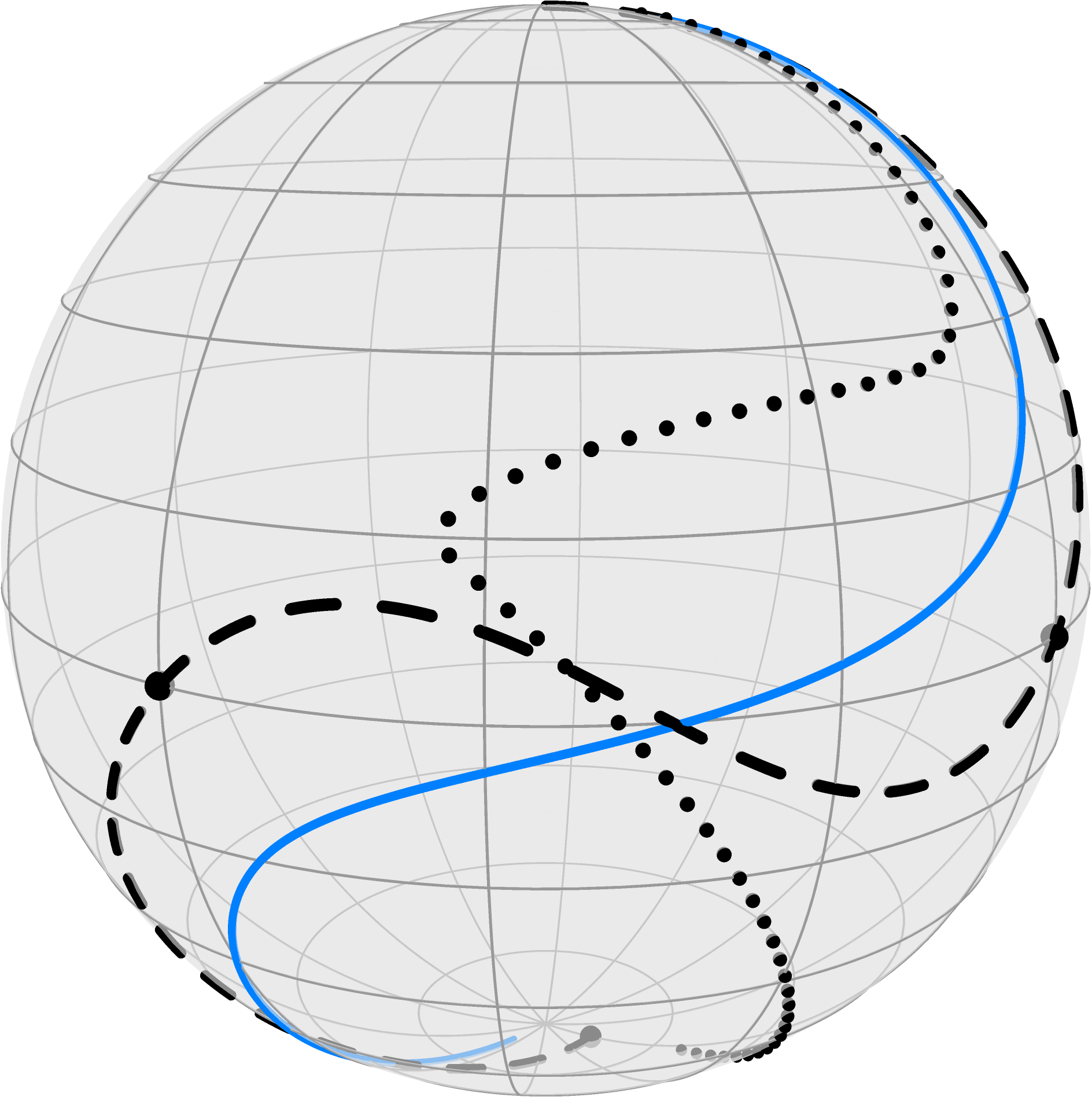}
      \caption{Bézier curves}
      \label{subfig:S2Compare:Curves}
    \end{subfigure}
    \begin{subfigure}[b]{.535\textwidth}\centering
    \hspace*{-1em}
      \includegraphics{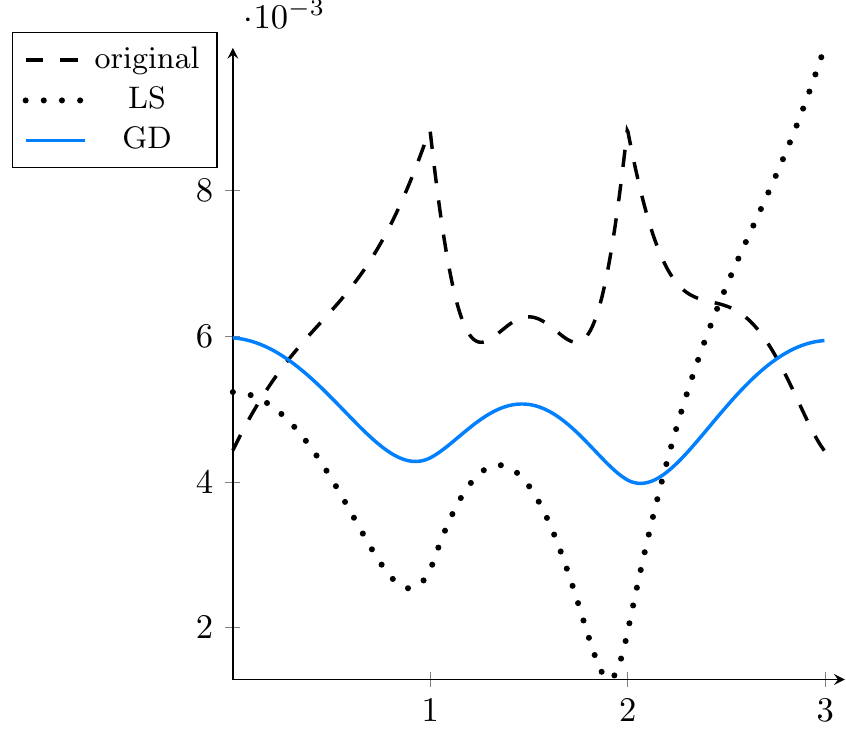}
      \caption{First order differences.}
      \label{subfig:S2Compare:firstOrder}
    \end{subfigure}
    \caption{
      The initial composite B\'ezier curve on $\mathbb{S}^2$
      (dashed, \subref{subfig:S2Compare:Curves}) has an MSA of~\(10.9103\).
      The curve obtained with the
      linear system (LS, dotted) of~\cite{Arnould2015} has
      an MSA of~\(7.3293\). The solution returned by our proposed
      method (GD, solid blue) outperforms them all with an MSA of~\(2.7908\).
      This is further illustrated by the first order
      derivative approximated by first order differences
      in~(\subref{subfig:S2Compare:firstOrder}).}
    \label{fig:S2Compare}
  \end{figure}

\subsection{An example of orientations}
Finally we compare our method with the blended splines introduced in~\cite{Gousenbourger2018}
for orientations, \ie, data given on $\mathrm{SO}(3)$.
Let
\begin{equation*}
  R_{\text{xy}}(\alpha) \!=\! 
  \begin{bmatrix}
    \cos\alpha&\sin\alpha&0\\
    -\sin\alpha&\cos\alpha&0\\
    0&0&1
  \end{bmatrix}\!\!,\ R_{\text{xz}}(\alpha) \!=\!
  \begin{bmatrix}
    \cos\alpha&0&\sin\alpha\\
    0&1&0\\
    -\sin\alpha&0&\cos\alpha
  \end{bmatrix}\!\!,\ R_{\text{yz}}(\alpha) \!=\!
  \begin{bmatrix}
    1&0&0\\
    0&\cos\alpha&\sin\alpha\\
    0&-\sin\alpha&\cos\alpha
  \end{bmatrix}
\end{equation*}
denote the rotation matrices in the $x-y$, $x-z$, and $y-z$ planes,
respectively. 
We introduce the three data points
\begin{align*}
  d_0 &= R_{\text{xy}}\left(\frac{4\pi}{9}\right)R_{yz}\left(-\frac{\pi}{2}\right), \\
  d_1 &= R_{\text{xz}}\left(-\frac{\pi}{8}\right)R_{\text{xy}}\left(\frac{\pi}{18}\right)R_{yz}\left(-\frac{\pi}{2}\right),\\
  d_2 &= R_{\text{xy}}\left(\frac{5\pi}{9}\right)R_{yz}\left(-\frac{\pi}{2}\right).
\end{align*}
These data points are shown in the first line of Fig.~\ref{fig:SO3Approx} in cyan.

We set $\lambda=10$ and discretize~\eqref{eq:EB} with \(N=401\) equispaced points.
This sampling is also
used to generate the first order finite differences. The parameters of the
gradient descent algorithm are set to the same values as for the examples on the sphere.

We perform the blended spline fitting with two 
segments and cubic splines. The resulting control points are shown in the second
row of Fig.~\ref{fig:SO3Approx}, in green. The objective value is
$\tilde A(\vect{b}) = 0.6464$.
We improve this solution with the minimization problem~\eqref{eq:ApproxFct} on
the product manifold $\operatorname{SO}(3)^{6}$.
We obtain the control points shown in the last line of 
Fig.~\ref{fig:SO3Approx} and an objective value of $\tilde A(\hat{\vect{b}}) = 0.2909$ 
for the resulting minimizer $\hat{\vect{b}}$.

We further compare both curves by looking at their absolute first order differences.
In the third line of Fig.~\ref{fig:SO3Approx}, we display $17$ orientations
along the initial curve with its first order difference as height. The line without
samples is the absolute first order differences of the minimizer $\hat{\vect{b}}$,
scaled to the same magnitude. The line is straightened especially for the first
B\'ezier segment.
Nevertheless, the line is still bent a little bit and
hence not a geodesic. This can be seen in the fourth line which represents
$17$ samples of the minimizing curve compared to its control points in the last
line, which are drawn on a straight line.

\begin{figure}[tbp]
  \centering
  \includegraphics[width=\textwidth]{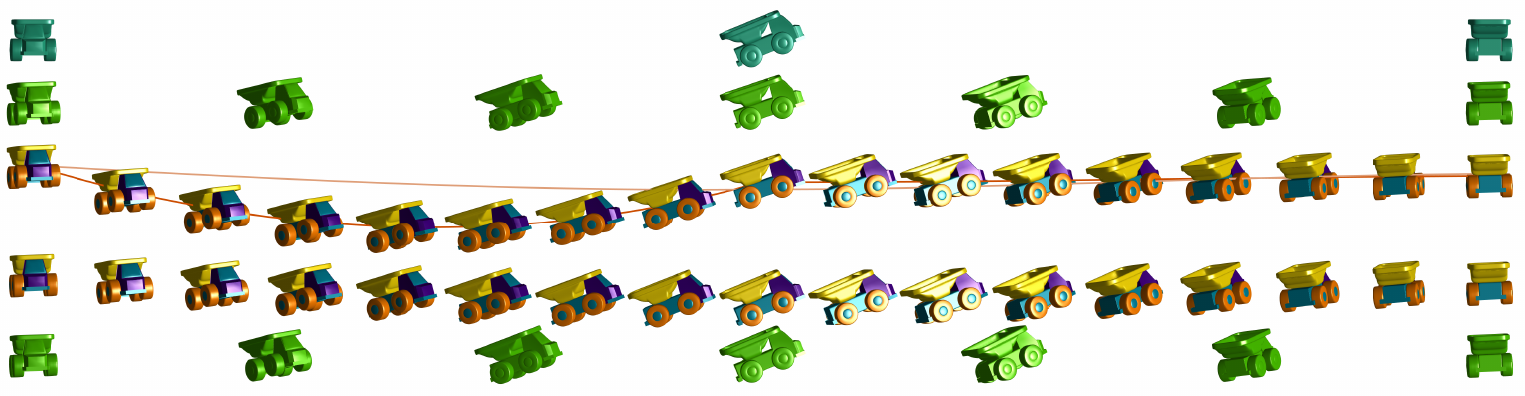}
  \caption{From top row to bottom row: 
  (1) the initial data points (cyan), 
  (2) the control points computed by the blended B\'ezier approach from~\cite{Gousenbourger2018},
  (3) 17 points on the corresponding curve, where the height represents the absolute first
  order difference, and the comparing curve is the first order differences from our
  model, 
  (4) 17 points along the resulting curve from gradient descent, and
  (5) the resulting control points of the curve in (4).}
  \label{fig:SO3Approx}
\end{figure}
\section{Conclusion and future work}\label{sec:Concl}
  In this paper, we introduced a method to solve the curve fitting
  problem to data points on a manifold, using composite B\'ezier curves.
  We approximate the mean squared acceleration of the curve 
  by a suitable second order difference and a trapezoidal rule, 
  and derive the corresponding gradient with respect to its control points.
  The gradient is computed in closed form by exploiting the recursive structure
  of the De Casteljau algorithm. Therefore, we obtain
  a formula that reduces to a concatenation of adjoint Jacobi fields.
  
  The evaluation of Jacobi fields is the only additional requirement compared
  to previous methods, which are solely evaluating exponential and logarithmic maps.
  For these, closed forms are available on symmetric manifolds.
  
  On the Euclidean space our solution reduces to the natural smoothing spline, the unique
  acceleration minimizing polynomial curve.
  The numerical experiments further confirm that the method presented in this paper 
  outperforms the tangent space(s)-based approaches with respect to the functional value.
  
  It is still an open question whether there exists a second order absolute
  finite difference on a manifold, that is jointly convex. Then convergence would follow by standard
  arguments of the gradient descent. For such a model, another interesting
  point for future work is to find out whether only an approximate evaluation of the Jacobi fields
  suffices for convergence. This would mean that, on manifolds where the Jacobi field
  can only be evaluated approximately by solving an ODE,
  the presented approach would still converge.

% Bibliography
\subsection*{Acknowledgements}
  This work was supported by the following fundings: RB gratefully acknowledges
  support by DFG grant BE~5888/2--1; PYG gratefully acknowledges support by
  the Fonds de la Recherche Scientifique -- FNRS and 
  the Fonds Wetenschappelijk Onderzoek -- Vlaanderen under EOS Project no 30468160, 
  and ``Communauté française de Belgique - Actions de Recherche
  Concertées'' (contract ARC 14/19-060).
  Fig.~\ref{subfig:S2-3Seg:IP} is based on an idea of Rozan I.~Rosandi.

\printbibliography
\end{document}